\documentclass[11pt]{amsart}
\usepackage{cite}
\usepackage{wrapfig, lipsum,booktabs}
\usepackage{graphicx}
\usepackage{amssymb}
\usepackage{epstopdf}
\usepackage{verbatim}
\usepackage{bm}
\usepackage{multicol}
\usepackage{multirow}
\usepackage{subfigure}
\usepackage{fancyhdr} 
\usepackage{float}
\usepackage{stmaryrd}
\usepackage{color}
\usepackage{soul}
\usepackage{tikz}
\usetikzlibrary{patterns}
\usepackage{pgfplots}
\usepackage{cancel}
\newtheorem{theorem}{Theorem}[section]

\newtheorem{remark}{Remark}[section]

\usepackage{multirow}
\usepackage{amsmath,amssymb,eucal}
\usepackage{graphicx,subfigure,epsfig}
\usepackage{psfrag}
\usepackage{url}
\usepackage[top=1in, bottom=1.in, left=1in, right=1in]{geometry}
\newcommand{\bld}[1]{\hbox{\boldmath$#1$}}    
\newcommand{\Th}{\mathcal{T}_h}
\newcommand{\Tho}{\mathcal{T}_h^t}
\newcommand{\Eh}{\mathcal{E}_h}
\newcommand{\Eho}{\mathcal{E}_h^t}
\newcommand{\HDG}{$\mathsf{HDG}$}
\newcommand{\TaH}{$\mathsf{TH}$}
\newcommand{\SV}{$\mathsf{SV}$}

\setulcolor{red}

\begin{document}
\title[ALE-HDG]{Arbitrary Lagrangian-Eulerian  hybridizable  discontinuous Galerkin methods for 
incompressible flow with moving boundaries and
interfaces}
\author{Guosheng Fu}
\address{Department of Applied and Computational Mathematics and 
Statistics, University of Notre Dame, USA.}
\email{gfu@nd.edu}
% \thanks{The author gratefully acknowledges the partial support of this work
% under AFOSR contract FA9550-12-1-0399.}

\begin{abstract}
We present
a class of Arbitrary Lagrangian-Eulerian
 hybridizable discontinuous Galerkin methods for the incompressible flow with 
moving boundaries and interfaces including two-phase flow with surface
tension.
\end{abstract}
\maketitle

\section{Introduction}
\label{sec:intro}
%Incompressible flow phenomena arise in numerous disciplines in science and engineering. 
%The finite element methods 
%have been widely used for the simulation of 
%incompressible flow since the 1970s; see the monographs
%\cite{GiraultRaviart86,Glowinski03, Layton08,BoffiBrezziFortin08,Elman14,John16}.
%In the last coupled of years, 
%there is a great interest in the use of 
%finite element methods for incompressible flow that yield 
%an exactly divergence-free velocity approximation 
%({\it strong mass conservation}), see the recent review
%article \cite{JohnLinkeMerdonNeilanRebholz17}.
%Stongly mass conserving finite element methods usually 
%enjoy the property of  {\it pressure-robustness} 
%\cite{JohnLinkeMerdonNeilanRebholz17}, which make them more accurate 
%than {\it non-pressure-robust} methods for certain flow regimes, e.g. 
%when the momentum balance is dominated
%by forces of a gradient type \cite{LinkeRebholz19}.

Incompressible flow problems with moving boundaries and interfaces
appear
naturally in a number of practical applications such as 
free surface flow, multi-phase flow, and fluid-structure interactions (FSI) \cite{Gross11,Richter17}. 
Using the terminology and categorization used in \cite{Tezduyar98}, a method
for flows with moving boundaries can be an interface-tracking (boundary-fitted) 
method or an interface-capturing (non-boundary-fitted) method, or possibly a combination of
the two. Examples of boundary-fitted methods include the 
arbitrary Lagrangian-Eulerian (ALE) method 
\cite{HughesLiuZimmermann81,Donea82, Donea04} and the space-time
method \cite{Tezduyar92,Tezduyar92b,MasudHughes97,vanderVegt08}, 
examples of non-boundary-fitted methods 
include the immersed boundary method \cite{Peskin02, Mittal05}, 
the immersed finite element
method \cite{Li98,ZhangGerstenberger04}, the fictitious domain method 
\cite{Glowinski01}, and the extended/generalized finite element method 
\cite{Belytschko03,FriesBelytschko10}, among many others.

In this paper, we present novel finite element methods for incompressible
flow with moving boundaries and interfaces using the ALE framework.
Our spatial discretization is based on a 
novel hybridizable discontinuous Galerkin (HDG) formulation
that may produce an exactly divergence-free velocity approximation;
see \cite{Cockburn16} for a review of HDG methods.
HDG methods that yield an exactly divergence-free velocity
approximation have been extensively studied for incompressible 
flow on static meshes, see, e.g., 
\cite{KonnoStenberg12,CockburnSayas14,
LehrenfeldSchoberl16, FuJinQiu19,RhebergenWells18}.
The extension of these schemes to moving domain problems was 
recently studied by Horvath and Rhebergen 
\cite{HorvathRhebergen19a,HorvathRhebergen19b} within the space-time
framework, and by Neunteufel \cite{Neunteufel17}
within the ALE  framework.
In this paper, we present a novel ALE-HDG scheme that is 
computationally more efficient than 
those in \cite{HorvathRhebergen19a,HorvathRhebergen19b,Neunteufel17}
due to the use of a novel set of hybrid unknowns, which includes 
the tangential component of velocity ({\it TV}) and normal-normal component of 
the stress ({\it NNS}) on the mesh skeleton. We name the new HDG scheme the {\it
TVNNS-HDG} scheme.

ALE-based finite element methods have been widely used 
for moving interface incompressible flow problems, see the 
recent review \cite{Tezduyar18}. 
The first ALE-based HDG method was introduced in 2016 by 
Sheldon et. al. \cite{Sheldon16} for FSI problems.
However, the spatial discretization for fluids 
in \cite{Sheldon16} was based on the  HDG scheme 
\cite{NguyenPeraireCockburnHDGStokes10} for Stokes flow which does not produce an
exactly divergence-free velocity approximation.
In 2017, Neunteufel \cite{Neunteufel17} introduced the 
ALE-$H(\mathrm{div})$-conforming-HDG
method for moving domain incompressible flows and FSI
in his diploma thesis.
However, due to the use of 
{\it Piola} mapping needed for the $H(\mathrm{div})$-conforming velocity finite
element space, the 
ALE-$H(\mathrm{div})$-conforming-HDG
formulation \cite{Neunteufel17} is significantly more complex than
ALE schemes based on 
classical continuous-velocity-based finite element methods like the 
Taylor-Hood element.
In this paper, we present a novel HDG scheme termed TVNNS-HDG 
that is based on completely discontinuous finite element spaces 
using the standard {\it pull-back} mappings, which
produces a divergence-conforming and divergence-free velocity approximation when the underlying mesh consists of affine simplices.
The $H(\mathrm{div})$-conformity of the scheme is achieved via 
the classical hybridization technique 
\cite{ArnoldBrezzi85,Cockburn04}, where we first relax the
$H(\mathrm{div})$-conformity of the velocity space then weakly impose it
back via a Lagrange multiplier (which is the normal-normal-stress variable).
Our spatial discretization on conforming affine simplicial meshes
is mathematically equivalent to the divergence-free HDG scheme
proposed back in 2010 by Lehrenfeld \cite{Lehrenfeld10}.
Hence, on static conforming simplicial meshes,  it readily enjoys properties such 
as high order accuracy and optimal convergence, 
important global and local conservation properties, energy-stability,
pressure robustness, a minimal amount of numerical dissipation and
computational efficiency \cite{LehrenfeldSchoberl16, Schroeder19}.

The rest of the paper is organized as follows.
In Section \ref{sec:ale}, we first introduce the TVNNS-HDG scheme for 
the steady-state incompressible Stokes equations, 
then apply it to the moving domain Navier-Stokes equations using the ALE
framework.
In Section \ref{sec:twophase}, we extend the ALE-TVNNS-HDG scheme 
to incompressible two-phase flow with surface tension.
Numerical results are presented in Section \ref{sec:num}.
We conclude in Section \ref{sec:conclude}.

\section{The ALE-TVNNS-HDG scheme for moving domain Navier-Stokes equations}
\label{sec:ale}
\subsection{The ALE-Navier-Stokes equations}
Consider the Navier-Stokes equation 
on a smooth-varying moving domain $\Omega^t\subset\mathbb{R}^d$, 
$d\in\{2,3\}$, for $t\in [0, T]$, 
given by a smooth ALE map \cite{Nobile01,Quaini08}: 
\begin{align}
  \label{ale}
  \mathcal{A}_t:\Omega^0\subset \mathbb{R}^d\longrightarrow \Omega^t, 
  \quad\quad \mathbf x(\mathbf x_0, t) = \mathcal{A}_t(\mathbf x_0),\quad  
  \forall t\in [0,T],
\end{align}
where the initial configuration $\Omega^0$ at $t=0$ is considered as the reference
configuration. 
The Navier-Stokes equations in ALE non-conservative form 
\cite{Nobile01,Quaini08} is given as follows: 
\begin{subequations}
\label{ns-eq-ale}
  \begin{alignat}{2}
\label{ns-eq-ale1}
  \rho\left.\frac{\partial \bld u}{\partial t}\right|_{\mathbf x_0} 
  +\rho(\bld u-\bld \omega)\cdot\nabla_{\mathbf x}\bld u 
  -\mathrm{div}_{\mathbf x}(2\mu\mathbf{D_x}(\bld u)  -p\bld {I})
    =&\; \rho\bld f, \quad&& \text{in} \;\Omega^t\times [0,T]\\
\label{ns-eq-ale2}
    \mathrm{div}_{\mathbf x}\bld u = &\;0,\quad     && \text{in} \;\Omega^t\times [0,T]
    %\quad in \;\;\Omega,
\end{alignat}
\end{subequations}
where $\mathbf D_{\mathbf x}$ is the symmetric strain rate tensor
\begin{align*}
  \mathbf {D_x}(\bld u) = \frac12(\nabla_{\mathbf x} \bld u + 
  (\nabla_{\mathbf x} \bld u)^T),
\end{align*}
$\bld I$ is the identity tensor, 
$\bld u(\mathbf x,t)$ is the velocity field, $p(\mathbf x, t)$ is the pressure, 
$\rho$ is the (constant) fluid density, $\mu$ is the (constant) coefficient of 
dynamic viscosity, $\bld f$ is the body forces, and
\[
\bld \omega(\mathbf x, t) =
  \left.\frac{\partial \mathbf x}{\partial
  t}\right|_{\mathbf x_0} = 
  \frac{\partial \mathcal A_t}{\partial
  t}\circ\mathcal{A}_t^{-1}(\mathbf x) 
\]
denotes the domain velocity. 
Throughout this section, we assume that ALE map \eqref{ale}
is given, although in most applications it represents a further unknown of the
problem.

\subsection{The TVNNS-HDG scheme for the steady-state Stokes equations}
In order to present our HDG discretization for the equations
\eqref{ns-eq-ale}, we first present the 
TVNNS-HDG scheme for the following 
steady-state Stokes equations on 
the domain $\Omega^t$, with a fixed time $t\in[0,T]$:
\begin{alignat}{2}
  \label{stokes}
  \Big\{
  \begin{split}
    -\mathrm{div}_{\mathbf x}
    (2\mu\mathbf{D}_{\mathbf
  x}(\bld u)  -p\bld {I})
     =&\; \rho\bld f,\\
    \mathrm{div}_{\mathbf x}\bld u = &\;0,
  \end{split}\quad&& \text{in} \;\Omega^t, 
  \quad\quad \bld u =\;\bld 0\; \text{on} \;\partial\Omega^t,
  \end{alignat}
  where the stress ${\bld \sigma}:= 
    2\mu\mathbf{D}_{\mathbf
 x}(\bld u)  -p\bld {I}$.
We assume that the domain $\Omega^t$ is obtained from the reference domain
$\Omega^0$ by the mapping \eqref{ale}. 
Let $\Th^0:=\{T^0\}$ be a conforming 
simplicial triangulation of $\Omega^0$, and let $\Eh^0$ be the set of facets of
$\Th^0$.
We denote $\Tho:=\{\mathcal{A}_t(T^0)\}$ as the mapped triangulation on 
  the domain $\Omega^t$, 
 and denote  $\Eho$  as the set of facets of $\Tho$.
 We set $h$ to be the maximum mesh size of $\Tho$.
Given an affine simplex $\widehat{S}\subset \mathbb{R}^d$, $d=1,2,3$, we denote 
$\mathcal{P}^m(\widehat{S})$ as the space of polynomials of degree at most $m$ on the mesh for any integer $m\ge 0$. Given any mapped simplex $S=\mathcal{A}(\widehat{S})$, with an abuse of notation, we also denote
\[
\mathcal{P}^m({S}):=\{v\in L^2(S):\quad 
v= \widehat{v}\circ\mathcal{A}^{-1}
 \text{ with }  \widehat{v}|_{\widehat{S}}\in \mathcal{P}^k(\widehat{S})\}.
\]
Note that the standard 
{\it pull-back} transformation 
$\bullet(\mathbf x, t) = \widetilde{\bullet}\circ\mathcal{A}^{-1}(\mathbf x)$
is used to define the space $\mathcal{P}^m(S)$.
We now define the following discontinuous polynomial finite element spaces based on pull-back transformations:
\begin{subequations}
  \label{fes}
  \begin{alignat}{2}
    \label{fes-v}
    \bld V_h^t :=&\;\{\bld v\in [L^2(\Tho)]^d:&&\quad 
 \bld v|_{T}\in [\mathcal{P}^k(T)]^d,\quad \forall T\in\Tho
  % \bld v=\bld v^0\circ\mathcal{A}_t^{-1}, 
%  \text{ with }  \bld v^0|_{T^0}\in [\mathcal{P}^k(T^0)]^d,\quad \forall T^0\in\Th^0
\},\\
    \label{fes-q}
      Q_h^t :=&\;\{q\in L^2(\Tho):&&\quad 
 q|_{T}\in \mathcal{P}^{k-1}(T),\quad \forall  T\in\Tho
%        q=q^0\circ\mathcal{A}_t^{-1}, \text{ with }
 %   q^0|_{T^0}\in \mathcal{P}^{k-1}(T^0),\quad \forall  T^0\in\Th^0
\},\\
    \label{fes-vhat}
        \widehat{\bld V}_h^t :=&\;\{\widehat{\bld v}\in [L^2(\Eho)]^d:&&\quad 
       \widehat{\bld v}=
         \oplus_{j=1}^{d-1}(\widehat{v}_j)\bld t_j, \text{
          with }
         \widehat{v}_j|_{F}\in
    \mathcal{P}^{k}(F), \quad \forall F\in\Eho
%          \widehat{\bld v}=
  %        \oplus_{j=1}^{d-1}(\widehat{v}_j^0\circ\mathcal{A}_t^{-1})\bld t_j, \text{
    %      with }
   %       \widehat{v}_j^0|_{F^0}\in
   % \mathcal{P}^{k}(F^0), \quad \forall F^0\in\Eh^0
\},\\
    \label{fes-vhat0}
          \widehat{\bld V}_{h,0}^t :=&\;\{\widehat{\bld v}\in\widehat{\bld
            V}_h^t  :&&\quad 
  \widehat{\bld v}|_{F} = \bld 0, \quad \forall F\in\Eho\cap\partial \Omega^t 
\},\\
    \label{fes-qhat}
            \widehat{M}_h^t :=&\;\{\widehat{\tau}^{nn} \in L^2(\Eho):&&\quad 
            \widehat{\tau}^{nn}|_{F}\in \mathcal{P}^{k}(F),
  \quad \forall F\in\Eho
%              \widehat{\tau}^{nn}=
 %             \widehat{\tau}^{nn,0}\circ
  %            \mathcal{A}_t^{-1}, \text{ with }
   %           \widehat{ \tau}^{nn,0}|_{F^0}\in \mathcal{P}^{k}(F^0),
 % \quad \forall F^0\in\Eh^0 
\},
\end{alignat}
\end{subequations}
where 
$\{\bld t_j\}_{j=1}^{d-1}$ are the orthogonal tangential directions on the mesh
skeleton $\Eho$. 
The weak formulation of the TVNNS-HDG scheme for \eqref{stokes} is given as follows: 
Find 
%local unknowns
$(\bld u_h, p_h,
\widehat{\bld u}_h, \widehat{\sigma}_h^{nn})\in 
\bld 
V_h^t\times Q_h^t\times 
\widehat{\bld V}_{h,0}^t\times \widehat{M}_h^t$ 
  such that 
\begin{subequations}
  \label{hdg}
  \begin{align}
  \label{hdg1}
 2\mu\,     \mathcal{B}_h\left((\bld u_h,  
\widehat{\bld u}_h), 
    (\bld v_h, 
    \widehat{\bld v}_h)\right) 
    -
    \mathcal{D}_h\left(\bld v_h, 
    (p_h, 
  \widehat{\sigma}^{nn}_h)\right)
    &=\; {f}_h(\bld v_h),\\
  \label{hdg2}
    \mathcal{D}_h\left(\bld u_h, 
    (q_h, 
  \widehat{\tau}^{nn}_h)\right) &=\; 0,
  \end{align}
\end{subequations}
  for all 
$(\bld v_h, q_h,
\widehat{\bld v}_h, \widehat{\tau}_h^{nn})\in 
\bld 
V_h^t\times Q_h^t\times 
\widehat{\bld V}_{h,0}^t\times \widehat{M}_h^t$, 
  where
  \begin{subequations}
    \label{bilinearform} 
  \begin{align}
    \mathcal{B}_h\left((\bld u_h,  
    \widehat{\bld u}_h), (\bld v_h, \widehat{\bld v}_h)\right):=&\;
\sum_{T\in\Tho} \int_T \mathbf{D_x}(\bld u_h):\mathbf{D_x}(\bld
v_h)\;\mathrm{dx}
    - \int_{\partial T}\mathbf{D_x}(\bld u_h)\bld n
    \cdot \mathsf{tang}(\bld v_h-\widehat{\bld v}_h)
    \mathrm{ds}\\
                                                                &\hspace{-19ex}
-\underbrace{\int_{\partial T}\mathbf{D}_{\mathbf x}(\bld v_h)\bld n
    \cdot \mathsf{tang}(\bld u_h-\widehat{\bld u}_h)\mathrm{ds}}_{\text{ for
    symmetry}} 
    +\underbrace{\int_{\partial T}\frac{\alpha (k+1)^2}{h} 
    \mathsf{tang}(\bld u_h-\widehat{\bld u}_h)
    \cdot \mathsf{tang}(\bld v_h-\widehat{\bld v}_h)\mathrm{ds}}_{\text{ for
    stability}},\nonumber\\
    \mathcal{D}_h\left(\bld u_h, (q_h, \widehat{\tau}^{nn}_h)\right):=&\;
     \sum_{T\in\Tho}
     \int_T\mathrm{div}_{\mathbf x} (\bld u_h) q_h \mathrm{dx}
     +\int_{\partial T}
     \bld u_h\cdot\bld n\; \widehat{\tau}_h^{nn} \mathrm{ds},\\
    f_h(\bld v_h) := &\;
   \sum_{T\in\Tho} \int_T \rho\bld f\cdot \bld v_h\;\mathrm{dx},
  \end{align}
  \end{subequations}
  where $\bld n$ is the normal direction,   
$
  \mathsf{tang}(\bld w) := \bld w-(\bld w\cdot\bld n)\bld n
$
denotes the {\it tangential component} of a vector $\bld w$,
 the variable 
 $\widehat{\sigma}_h^{nn}$ approximates the normal-normal component of the
 stress $({\bld \sigma}\bld n)\cdot\bld n$, 
and  $\alpha>0$ is a 
sufficiently large stabilization constant taken to be $4$ in all our
numerical simulations.
%We refer to the new scheme \eqref{hdg} as  a {\it
%tangential-velocity-normal-normal-stress} based HDG scheme
%since the global unknows (after static condensation) consist of tangential component 
%of the
%velocity $\widehat{\bld u}_h$ and the normal-normal component of the 
%stress $\widehat{\sigma}_h^{nn}$ on the mesh skeleton.

The following result shows that 
the scheme \eqref{hdg} on affine simplicial meshes is mathematically equivalent to the divergence-free HDG scheme
proposed in \cite[Equation (2.3.10)]{Lehrenfeld10}.
Hence, optimal (pressure-robust) 
velocity error estimate that is  independent of 
the pressure regularity can be obtained for the scheme \eqref{hdg}, see 
\cite[Lemma 2.3.13--2.3.14]{Lehrenfeld10}.
\begin{theorem}\label{thm:eqv}
Assume the conforming simplicial mesh $\Tho$ is affine.
Let 
$(\bld u_h, p_h,
\widehat{\bld u}_h, \widehat{\sigma}_h^{nn})\in 
\bld 
V_h^t\times Q_h^t\times 
\widehat{\bld V}_{h,0}^t\times \widehat{M}_h^t$ be the solution to the scheme 
\eqref{hdg}. 
Then, 
$(\bld u_h, p_h, \widehat{\bld u}_h)$ is the solution to the following 
divergence-free HDG scheme proposed in 
\cite[Equation (2.3.10)]{Lehrenfeld10}:
Find 
$(\bld u_h, p_h, \widehat{\bld u}_h)\in 
(\bld V_h^t\cap H_0(\mathrm{div},\Omega^t))
\times Q_h^t\times 
\widehat{\bld V}_{h,0}^t$ such that 
  \begin{align*}
 2\mu\,     \mathcal{B}_h\left((\bld u_h,  
\widehat{\bld u}_h), 
    (\bld v_h, 
    \widehat{\bld v}_h)\right) 
    -
     \sum_{T\in\Tho}
     \int_T\mathrm{div}_{\mathbf x} (\bld v_h) p_h \mathrm{dx}
    &=\; {f}_h(\bld v_h),\\
     \sum_{T\in\Tho}
     \int_T\mathrm{div}_{\mathbf x} (\bld u_h) q_h \mathrm{dx}
    &=\; 0,
  \end{align*}
  for all 
$(\bld v_h, q_h, \widehat{\bld v}_h)\in 
(\bld V_h^t\cap H_0(\mathrm{div},\Omega^t))
\times Q_h^t\times 
\widehat{\bld V}_{h,0}^t$. 
\end{theorem}
\begin{proof}
Since the mesh consists of affine simplicies, equation \eqref{hdg2} implies
  that $\bld u_h$ is divergence-conforming and divergence-free, and 
  its normal component vanishes on the boundary. 
Taking  test functions
$(\bld v_h, q_h, \widehat{\bld v}_h)\in 
(\bld V_h^t\cap H_0(\mathrm{div},\Omega^t))
\times Q_h^t\times 
\widehat{\bld V}_{h,0}^t$ for the scheme \eqref{hdg}, we obtain the above equations
in Theorem \ref{thm:eqv}. This completes the proof.
\end{proof}
\begin{remark}
  \label{rk:extension}
The formulation \eqref{hdg} was briefly mentioned in 
\cite[Remark 7]{LehrenfeldSchoberl16}, but it has not been seriously 
compared with the divergence-conforming velocity space-based formulation
\cite{Lehrenfeld10, LehrenfeldSchoberl16}.
We show in the next subsection
that due to the use of completely discontinuous finite element
space and pull-back mappings for the velocity approximation, the extension of the 
scheme \eqref{hdg} to the ALE Navier-Stokes equations \eqref{ns-eq-ale}
poses no extra difficulties.
%leads to a computationally much more efficiently scheme than that for the divergence-conforming velocity space-based 
%formulation \cite{Neunteufel17}.
\end{remark}
\begin{remark}
  \label{rk:curve}
When the mesh $\Tho$ is curved, equation
\eqref{hdg2} would neither imply divergence-conformity or locally
divergence-free of the velocity approximation $\bld u_h$ due to the use of 
pull-back transformation in \eqref{fes-v} since, in general,  
$\bld u_h\cdot\bld n|_F \not\in 
\mathcal{P}^k(F)$
for curved facets $F$, and 
$\mathrm{div}_{\mathbf x}(\bld u_h)|_T \not\in 
\mathcal{P}^{k-1}(T)$
for curved elements $T$.
The analysis of the scheme \eqref{hdg} for the curved mesh case is more complicated, which  consists our
ongoing work. 
\end{remark}
\subsection{The ALE-TVNNS-HDG scheme}
Here we extend the HDG scheme \eqref{hdg}
to the ALE Navier-Stokes equations \eqref{ns-eq-ale}. 
In particular, we point out that: 
\begin{itemize}
\item The nonlinear 
convection term in \eqref{ns-eq-ale} will be 
discretized using a variant \cite{Lehrenfeld10} 
of the upwinding
technique \cite{CockburnKanschatSchotzau07} thanks to the use of 
discontinuous velocity finite element space, 
which provides a minimal amount of numerical dissipation
that seems to be  strong enough to  suppress unphysical oscillation and 
gives the scheme extra stability properties in 
the convection-dominated case; see, e.g., 
\cite{Schroeder18,Fu19}.
No additional {\it convection stabilization}
is required for the HDG formulation, which has to be contrasted with
methods using continuous velocity approximation, where 
extra stabilization terms \cite{BrooksHughes82} are needed for stability,
usually via a variational multiscale approach \cite{Hughes95}. 
We also mention that traditional
variational multiscale technique might pollute the divergence-free property of
the velocity approximation; see 
\cite[Section 2.2.3]{Kamensky17}, see also
the recent work on an isogeometric divergence-free variational multiscale formulation \cite{vanOpstal17}. 
\item
Thanks to the use of completely
discontinuous velocity finite element space $\bld V_h^t$
via the {pull-back} transformation, 
the time derivative term in \eqref{ns-eq-ale} can be easily treated 
similarly as the fixed domain case.
  \end{itemize}
  \subsubsection{Semidiscrete scheme}
%  To this end, we assume that the mesh $\Tho$ for the 
%  domain $\Omega^t$, for $t\in[0,T]$, is obtained via the ALE mapping
%  \eqref{ale} from a reference conforming simplicial mesh $\Th^0$ for the domain $\Omega_0$, i.e.,
%  \[
%    \Tho:=\{\mathcal{A}_t(T^0): \;\forall T^0\in \Th^0\}
%  .\]
For simplicity, 
we assume the equations \eqref{ns-eq-ale} is equipped with a homogeneous Dirichlet boundary condition for velocity on the whole
boundary $\partial \Omega^t$, for all $t\in[0,T]$.
%Furthermore, we define the following space-time (discontinuous) finite element space
%\begin{align}
%  \label{vh-st}
%  \bld V_h(t):=\{\bld \phi(\mathbf x, t):= 
%    \bld \phi^0\circ(\mathcal{A}_t)^{-1}(\mathbf x):\quad\forall 
%    \bld \phi^0\in \bld V_h^0
%\},
%\end{align}
%which will be used to discretize the velocity approximation.
%Note that for a given fixed $t\in [0,T]$, the space $\bld V_h(t)$ is equivalent to 
%$\bld V_h^t$ on $\Omega^t$.
Then, the semidiscrete ALE-divergence-free-HDG scheme reads as follows:
Find 
$(\bld u_h, p_h,
\widehat{\bld u}_h, \widehat{\sigma}_h^{nn})\in 
\bld 
V_h^t\times Q_h^t\times 
\widehat{\bld V}_{h,0}^t\times \widehat{M}_h^t$ 
  such that 
\begin{subequations}
  \label{ale-hdg}
  \begin{align}
  \label{ale-hdg1}
  \rho\,\mathcal{M}_h(\left.\frac{\partial \bld u_h}{\partial t}\right|_{\mathbf x_0},  \bld v_h)+
\rho\,\mathcal{C}_h\left((\bld \omega,\bld u_h, \widehat{\bld u}_h),  
(\bld v_h,\widehat{\bld v}_h)\right)\hspace{9ex}&\\
+
 2\mu\,     \mathcal{B}_h\left((\bld u_h,  
\widehat{\bld u}_h), 
    (\bld v_h, 
    \widehat{\bld v}_h)\right) 
    -
    \mathcal{D}_h\left(\bld v_h, 
    (p_h, 
  \widehat{\sigma}^{nn}_h)\right)
    &=\; {f}_h(\bld v_h),\nonumber\\
  \label{ale-hdg2}
    \mathcal{D}_h\left(\bld u_h, 
    (q_h, 
  \widehat{\tau}^{nn}_h)\right) &=\; 0,
  \end{align}
\end{subequations}
for all 
$(\bld v_h, q_h,
\widehat{\bld v}_h, \widehat{\tau}_h^{nn})\in 
\bld 
V_h^t\times Q_h^t\times 
\widehat{\bld V}_{h,0}^t\times \widehat{M}_h^t$, 
  where the mass operator
   \[ \mathcal{M}_h\left(\bld u_h, \bld v_h\right):=\;
\sum_{T\in\Tho} \int_T \bld u_h\cdot\bld
v_h\;\mathrm{dx},\]  and the nonlinear convection operator 
$
    \mathcal{C}_h\left((\bld\omega,\bld u_h, \widehat{\bld u}_h), 
  (\bld v_h, \widehat{\bld v}_h)\right)$ is given below:
{
  \begin{align*}
    \mathcal{C}_h
  :=&\;
     \sum_{T\in\Tho}
     -\int_T
     \bld u_h\cdot \mathrm{div}_{\mathbf x}((\bld u_h-\bld \omega)\otimes \bld
     v_h) \mathrm{dx}%\\
                                         %&\hspace{-0ex}
                                         +\int_{\partial T}
     (\bld u_h-\bld \omega)\cdot\bld n
    \, \widehat{\bld u}^{up}_h\cdot \mathsf{tang}(\bld
     v_h-\widehat{\bld v}_h) \mathrm{ds},
                                     \end{align*}}
with $\widehat{\bld u}^{up}_h$ is the following
  {\it upwinding} numerical flux for the tangential velocity component:
  \[
    \widehat{\bld u}^{up}_h:= 
    %(\bld u_h\cdot\bld n)\bld n + 
    \left\{
      \begin{tabular}{ll}
        $\mathsf{tang}(\bld u_h)$ & if $(\bld u_h-\bld \omega)\cdot\bld
        n\ge0$\\[.4ex]
        $\mathsf{tang}(\widehat{\bld u}_h)$ & if $(\bld u_h-\bld \omega)\cdot\bld
        n<0$
    \end{tabular}\right.
  .\]
  The ALE Navier-Stokes equations \eqref{ns-eq-ale} hold the following energy
  identity:
  \[
  \frac12\frac{d}{dt} \rho\|\bld u\|^2_{L^2(\Omega^t)}
  +
  2\mu\,\|\mathbf{D_x}(\bld u)\|^2_{L^2(\Omega^t)}
  =\rho \int_{\Omega^t}\bld f\cdot\bld u\,\mathrm{dx} 
  %- \frac12 \rho\int_{\Omega^t} |\bld
 %   u|^2\mathrm{div}_{\mathbf x} \bld \omega\,\mathrm{dx}
  .\] 
  The following result 
  shows that a similar energy identity holds for the  semidiscrete scheme
  \eqref{ale-hdg} provided that the mesh is affine.  
  %It can be proven using a combination of the standard
  %energy argument and the Reynolds transport theorem, which is 
  %left as an exercise for the reader. 
\begin{theorem}
  \label{thm:ener}
Assume mesh $\Th^t$ is affine.
  Let $(\bld u_h, p_h,
\widehat{\bld u}_h, \widehat{\sigma}_h^{nn})\in 
\bld 
V_h^t\times Q_h^t\times 
\widehat{\bld V}_{h,0}^t\times \widehat{M}_h^t$ 
be the solution to the equations \eqref{ale-hdg} for all $t\in[0,T]$. 
Then the following energy identity holds
\[
  \frac12\frac{d}{dt} \rho\|\bld u_h\|^2_{L^2(\Omega^t)}
  +
 2\mu\, \mathcal{B}_h\left((\bld u_h,  
\widehat{\bld u}_h), 
    (\bld u_h, 
    \widehat{\bld u}_h)\right) 
    + \mathsf{DISP}  = f_h(\bld u_h) 
    %- \frac12\rho \int_{\Omega^t} |\bld
    %u_h|^2\mathrm{div}_{\mathbf x} \bld \omega\,\mathrm{dx}
,\]
where 
$\mathsf{DISP}$ is the following
nonnegative numerical dissipation term:
\[
  \mathsf{DISP}:=\frac12 \rho\sum_{T\in\Tho}
  \int_{\partial T} |(\bld u_h-\bld \omega)\cdot\bld n|\;
  %(
  |\mathsf{tang}(\bld u_h-\widehat{\bld u}_h)|^2
  %+|\mathsf{tang}(\bld u_h^+-\widehat{\bld u}_h)|^2
  %)
  \,\mathrm{ds}.
\] 
\end{theorem}
\begin{proof}
  Taking the test functions  
$(\bld v_h, q_h,
\widehat{\bld v}_h, \widehat{\tau}_h^{nn})
:= (\bld u_h, p_h,
\widehat{\bld u}_h, \widehat{\sigma}_h^{nn})
$ and adding, we get 
\[
  \rho\,\mathcal{M}_h(\left.\frac{\partial \bld u_h}{\partial
    t}\right|_{\mathbf x_0},  \bld u_h)+
\rho\,\mathcal{C}_h\left((\bld \omega,\bld u_h, \widehat{\bld u}_h),  
(\bld u_h,\widehat{\bld u}_h)\right)
+
 2\mu\,     \mathcal{B}_h\left((\bld u_h,  
\widehat{\bld u}_h), 
    (\bld u_h, 
    \widehat{\bld u}_h)\right) 
    =\; {f}_h(\bld v_h).
\] 
 Using the Reynolds transport theorem, we have 
 \[
  \mathcal{M}_h(\left.\frac{\partial \bld u_h}{\partial
    t}\right|_{\mathbf x_0},  \bld u_h)
= 
  \frac12\frac{d}{dt} \|\bld u_h\|^2_{L^2(\Omega^t)}
  -\frac12\int_{\Omega^t} |\bld
    u_h|^2\mathrm{div}_{\mathbf x} \bld \omega\,\mathrm{dx}.
   \] 
   By equations \eqref{ale-hdg2}, we have 
   $\bld u_h$ is divergence-conforming and divergence-free.
   Hence, the nonlinear term 
$\mathcal{C}_h\left((\bld \omega,\bld u_h, \widehat{\bld u}_h),  
(\bld u_h,\widehat{\bld u}_h)\right)$ can be simplified as 
\[
  \frac12 \sum_{T\in\Tho}
  \int_{\partial T} |(\bld u_h-\bld \omega)\cdot\bld n|\;
  |\mathsf{tang}(\bld u_h-\widehat{\bld u}_h)|^2
  \,\mathrm{ds}
  +\frac12\int_{\Omega^t} |\bld
    u_h|^2\mathrm{div}_{\mathbf x} \bld \omega\,\mathrm{dx}.
\]
Combining the above two expressions, we arrive at the desired energy 
identity in  Theorem \eqref{thm:ener}.
\end{proof}
  \subsubsection{Fully implicit scheme}
Let  $\{\bld \phi_i^t:=\bld\phi^0_i \circ\mathcal{A}_t^{-1}
  \}_{i=1}^{N}$ be a set of 
  basis functions for $\bld V_h^t$ where 
  $\{\bld \phi_i^0\}_{i=1}^{N}$ is a set of basis functions 
  for $\bld V_h^0$. We observe that the following time derivative of
  the basis functions vanishes:
  \[
  \left. \frac{\partial{\bld \phi_i^t}}{\partial t}\right|_{\mathbf x_0}
    = \left. \frac{\partial{\bld \phi_i^0}}{\partial t}\right|_{\mathbf x_0}
      \circ\mathcal{A}_t^{-1} = 0.
  \]
Now,  denote  $\{t^n:=n\,\delta t\}_{n=0}^M$ be the uniform partition of the time
  interval $[0,T]$, with time step $\delta t$ and $T=M\delta t$.
  Let 
  \[
  \bld u_h^m : =\sum_{i=1}^N \bld \phi_i^{t^m}(\mathbf x) u_i^m\in \bld
V_h^{t^m}\]
    be the velocity approximation at time $t^m$. 
The {\it pull-back} of $\bld u_h^m$  to any other configuration $\Omega_s$ with 
is given by 
\[
\bld u_h^m\circ\mathcal{A}_s\circ(\mathcal{A}_{t^m})^{-1}
=\sum_{i=1}^N \bld \phi_i^{s}(\mathbf x) u_i^m\in \bld V_h^s,
\]
which, to lighten notation, we still denote as $\bld u_h^m$.
We use backward difference formula (BDF) to discretize the time derivative term
in \eqref{ale-hdg1}. So we introduce the backward discretization operators
(up to third order accuracy)
applied to the function $\bld u_h^m$:
\begin{alignat*}{2}
  D_t^1\bld u_h^m : =& \frac{1}{\delta t} (\bld u_h^m-\bld u_h^{m-1}),\quad && \text{
  BDF1},\\
  D_t^2\bld u_h^m : =& \frac{1}{\delta t} (\frac32\bld u_h^m-2\bld u_h^{m-1}
  +\frac12\bld u_h^{m-2}),\quad && \text{
  BDF2},\\
  D_t^3\bld u_h^m : =& \frac{1}{\delta t} (\frac{11}{6}\bld u_h^m-3\bld u_h^{m-1}
  +\frac32\bld u_h^{m-2}-\frac13\bld u_h^{m-3}),\quad && \text{
  BDF3}.
\end{alignat*}
Finally, the fully discrete scheme with BDF time discretization reads as follows:
Given $s\in\{1,2,3\}$ and $\bld u_h^j$ for $j\le s-1$, find, for $m\in\{s, \cdots, M\}$,  
$(\bld u_h^m, p_h^m,
\widehat{\bld u}_h^m, \widehat{\sigma}_h^{nn,m})\in 
\bld 
V_h^{t_m}\times Q_h^{t_m}\times 
\widehat{\bld V}_{h,0}^{t_m}\times \widehat{M}_h^{t_m}$ 
  such that 
\begin{subequations}
  \label{ale-hdg-full}
  \begin{align}
  \label{ale-hdg-full1}
  \rho\,\mathcal{M}_h(D_t^s \bld u_h^m,  \bld v_h)+
\rho\,\mathcal{C}_h\left((\bld \omega,\bld u_h^m, \widehat{\bld u}_h^m),  
(\bld v_h,\widehat{\bld v}_h)\right)\hspace{9ex}&\\
+
 2\mu\,     \mathcal{B}_h\left((\bld u_h^m,  
\widehat{\bld u}_h^m), 
    (\bld v_h, 
    \widehat{\bld v}_h)\right) 
    -
    \mathcal{D}_h\left(\bld v_h, 
    (p_h^m, 
  \widehat{\sigma}^{nn,m}_h)\right)
    &=\; {f}_h(\bld v_h),\nonumber\\
  \label{ale-hdg-full2}
    \mathcal{D}_h\left(\bld u_h^m, 
    (q_h, 
  \widehat{\tau}^{nn}_h)\right) &=\; 0,
  \end{align}
\end{subequations}
for all 
$(\bld v_h, q_h,
\widehat{\bld v}_h, \widehat{\tau}_h^{nn})\in 
\bld 
V_h^{t_m}\times Q_h^{t_m}\times 
\widehat{\bld V}_{h,0}^{t_m}\times \widehat{M}_h^{t_m}$.
We note that the fully discrete scheme results in a system of nonlinear
equations due to the implicit treatment of the nonlinear convection term.
It has been shown in \cite{Nobile01} that a fully discrete scheme is only 
{\it conditionally stable} even if it is based on BDF1 time discretization, 
with maximum allowable time step depends on the speed of domain deformation.
Hence, it makes sense to treat the (nonlinear) convection term
in \eqref{ns-eq-ale} explicitly to yield a less computational intensive,
{\it conditionally stable} linear scheme as we show next.
\subsubsection{Implicit-Explicit scheme} 
Here we treat the nonlinear convection term explicitly with a standard upwinding DG
formulation:
  \begin{align*}
    \mathcal{C}_h^{dg}(\bld w, \bld u_h, \bld v_h)
  :=&\;
     \sum_{T\in\Tho}
     -\int_T
     \bld u_h\cdot \mathrm{div}_{\mathbf x}((\bld u_h-\bld \omega)\otimes \bld
     v_h) \mathrm{dx}%\\
                                         %&\hspace{-0ex}
                                         +\int_{\partial T}
     (\bld u_h-\bld \omega)\cdot\bld n
     \, \mathsf{tang}({\bld u}^{-}_h)\cdot \mathsf{tang}(\bld
     v_h) \mathrm{ds},
\end{align*}
where ${\bld u}^{-}_h|_F = (\bld u_h|_{T^-})|_F$ 
is the {\it upwinding} numerical flux on the facet $F$ with $T^-$ the element
sharing the facet $F$ such that $(\bld u_h-\bld \omega)\cdot\bld n_{T^-}|_F >0$.
The implicit-explicit (IMEX) scheme with IMEX-SBDF 
\cite{AscherRuuthWetton93} temporal discretization reads as follows:
Given $s\in\{1,2,3\}$ and $\bld u_h^j$ for $j\le s-1$, find, for $m\in\{s, \cdots, M\}$,  
$(\bld u_h^m, p_h^m,
\widehat{\bld u}_h^m, \widehat{\sigma}_h^{nn,m})\in 
\bld 
V_h^{t_m}\times Q_h^{t_m}\times 
\widehat{\bld V}_{h,0}^{t_m}\times \widehat{M}_h^{t_m}$ 
  such that 
\begin{subequations}
  \label{ale-hdg-imex}
  \begin{align}
  \label{ale-hdg-imex1}
  \rho\,\mathcal{M}_h(D_t^s \bld u_h^m,  \bld v_h)+
  \rho\,\mathcal{C}_h^{dg}\left(\bld \omega,\widetilde{\bld u}_h^{m,s},\bld
  v_h\right)\hspace{16ex}&\\
+
 2\mu\,     \mathcal{B}_h\left((\bld u_h^m,  
\widehat{\bld u}_h^m), 
    (\bld v_h, 
    \widehat{\bld v}_h)\right) 
    -
    \mathcal{D}_h\left(\bld v_h, 
    (p_h^m, 
  \widehat{\sigma}^{nn,m}_h)\right)
    &=\; {f}_h(\bld v_h),\nonumber\\
  \label{ale-hdg-imex2}
    \mathcal{D}_h\left(\bld u_h^m, 
    (q_h, 
  \widehat{\tau}^{nn}_h)\right) &=\; 0,
  \end{align}
\end{subequations}
for all 
$(\bld v_h, q_h,
\widehat{\bld v}_h, \widehat{\tau}_h^{nn})\in 
\bld 
V_h^{t_m}\times Q_h^{t_m}\times 
\widehat{\bld V}_{h,0}^{t_m}\times \widehat{M}_h^{t_m}$, where 
\begin{alignat*}{2}
  \widetilde{\bld u}_h^{m,s} : =
  \left\{
    \begin{tabular}{ll}
      $\bld u_h^{m-1}$, & {for $s=1$},\\[.4ex]
      $2\bld u_h^{m-1}-\bld u_h^{m-2}$, & {for $s=2$},\\[.4ex]
  $3\bld u_h^{m-1}-3\bld u_h^{m-2}+\bld u_h^{m-3}$,\quad & {for $s=3$}.
    \end{tabular}
\right.
  \end{alignat*}
  The IMEX scheme \eqref{ale-hdg-imex} is linear and the resulting global
  linear system after static condensation is a symmetric Stokes-like system with global
  unknowns consists of tangential velocity and normal-normal stress on the mesh
  skeleton.
\begin{remark}
    When the flow problem \eqref{ns-eq-ale} is convection-dominated and the mesh is 
under-resolved
for the viscous term (which is typical the case for high Reynolds number flows), 
it further make sense to treat the second-order viscous term in
\eqref{ns-eq-ale} explicitly with a DG formulation \cite{CockburnKanschatSchotzau07} 
leaving only the velocity-pressure coupling implicit, see our recent studies \cite{Fu19,Fu19b} on
static meshes.  
\end{remark}

\section{The ALE-TVNNS-HDG scheme for two-phase flow}
\label{sec:twophase}
In this section, we extend the ALE-TVNNS-HDG scheme proposed in
Section \ref{sec:ale} to incompressible two-phase flow with
surface tension.
\begin{figure}[h!]
\centering
\includegraphics[width=.3\textwidth]{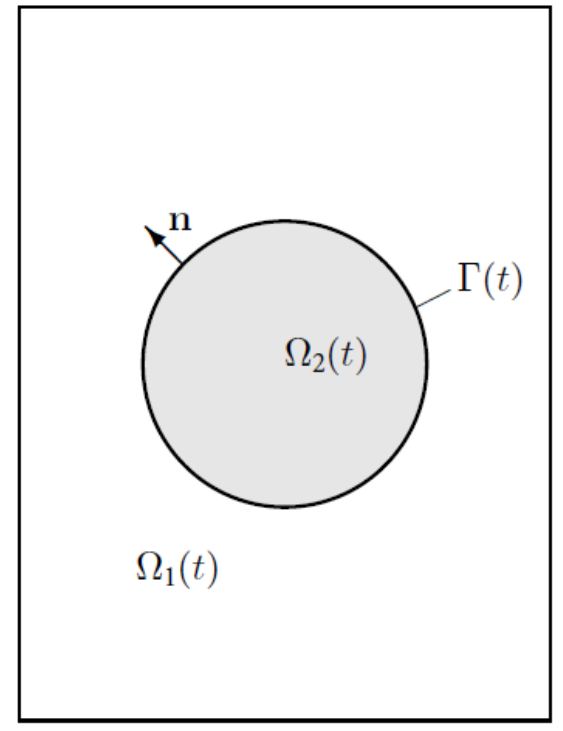}
  \caption{Schematic view of the two-phase domain.}
\label{fig:tp}
\end{figure} 
Assume a geometrical setting as sketched in 
Fig.~\ref{fig:tp}: denote by $\Omega$ a fixed domain such that 
$\overline{\Omega}=\overline{\Omega}_1(t)\cup\overline{\Omega}_2(t)$, 
where $\Omega_1(t)$, $\Omega_2(t)$ denote two domains separated by
a sharp smooth interface 
$\Gamma(t):=\overline{\Omega}_1(t)\cap\overline{\Omega}_2(t)$
with normal vector $\bld n$ pointing from 
$\Omega_2(t)$ into $\Omega_1(t)$.
Assume the domain $\Omega_i$ occupies an incompressible fluid with densitiy 
$\rho_i$ and dynamic viscosity $\mu_i$ for $i=1,2$.
Then, the {\it immiscible} incompressible two-phase  flow equations with
surface tension reads as follows:
Given an initial interface $\Gamma_0\subset \Omega$ and corresponding
sub-domains $\Omega_i(0)$, initial velocity $\bld u_0$ and 
volume force $\bld f$, find a velocity $\bld u(\mathbf{x}, t)$, 
pressure $p(\mathbf{x}, t)$, and interface $\Gamma(t)$ such that
$\bld u(\cdot, 0) = \bld u_0$, $\Gamma(0) = \Gamma_0$, and for all 
$t\in [0,T]$ the following equations hold,
\begin{subequations}
  \label{two-phase-eq}
  \begin{alignat}{2}
\label{two-eq-ale1}
\rho_i(\frac{\partial \bld u}{\partial t}
+\bld u\cdot\nabla_{\mathbf x}\bld u) 
  -\mathrm{div}_{\mathbf x}\bld \sigma_i
  %(2\rho_i\mu\mathbf{D_x}(\bld u)  -p\bld {I})
  =&\; \rho_i\bld f, \quad&& \text{in} \;\Omega_i, \quad
  i=1,2\\
\label{two-eq-ale2}
    \mathrm{div}_{\mathbf x}\bld u = &\;0,\quad     && \text{in}
    \;\Omega_i, \quad i=1,2,\\
\label{two-eq-ale3}
    \bld u = &\;0,\quad     && \text{on}
    \;\partial\Omega,\\
\label{two-eq-ale4}
    [\bld \sigma\bld n] = -\tau \kappa\bld n, 
                                     &\quad [\bld u] = 0, \quad &&\text{on}
                                     \;\Gamma,\\
\label{two-eq-ale5}
V_{\Gamma} =&
\;\bld u\cdot\bld n, \quad &&\text{on}\;
                                     \Gamma,
\end{alignat}
\end{subequations}
where
 ${\bld \sigma}_i:= 
    2\mu_i\mathbf{D}_{\mathbf
 x}(\bld u)  -p\bld {I}$ is the stress on domain $\Omega_i$,
 $\tau$ is the {\it surface tension coefficient} assumed to be
constant, the scalar function 
\[\kappa(\mathbf{x}):=\mathrm{div}_{\mathbf x}\bld n(\mathbf{x}), 
\quad \mathbf{x}\in\Gamma\]
is the  mean curvature, and $V_{\Gamma}=V_{\Gamma}(\mathbf x,
t)\in\mathbb{R}$ denotes the {\it size} of the
velocity of the interface $\Gamma$ at $\mathbf{x}\in\Gamma(t)$.

We solve the moving interface problem \eqref{two-phase-eq} on ALE moving interface-fitted meshes.
In particular, we combine the IMEX-ALE-TVNNS-HDG scheme \eqref{ale-hdg-imex} with a semi-implicit treatment \cite{Fernandez07} of  the domain motion where the interface is updated explicitly, which results in a {\it linear} scheme. 
We use the Laplace-Beltrami technique %\cite{XXX}
to approximate the mean curvature vector $\kappa\bld n$ on the discrete interface. Details of the fully discrete algorithm using second-order IMEX-SBDF time stepping  ($s=2$) is given below.

We start with an interface-fitted conforming simplicial mesh $\Th^0=\Th^{1,0}\cup \Th^{2,0}$ of the domain $\Omega$ where $\Th^{i,0}$ is a triangulation of $\Omega_i(0)$ for $i=1,2$, and $\Gamma_h^0:=\overline{\Th^{1,0}}\cap\overline{\Th^{2,0}}$ is a discretization of the interface $\Gamma_0$. 
Again, we denote  $\{t^n:=n\,\delta t\}_{n=0}^M$ be the uniform partition of the time
  interval $[0,T]$ with time step $\delta t$ and $T=M\delta t$. 
Given $\bld u_h^j\in\bld V_h^j$ as the velocity approximation at time $t^j$ with $j=0,1$, 
and let $\Th^1:=\mathcal{A}_h^1(\Th^0)$ be the mapped mesh at time $t^1$ with 
$\Gamma_h^1:=\mathcal{A}_h^1(\Gamma_h^0)$ the interface approximation at $t^1$.
Then,  for $m\in\{2,\cdots, M\}$ we perform the following four steps to update the computational mesh and field unknows at time $t^m$:
%\hfill\begin{minipage}{\dimexpr.95\textwidth}
  \begin{itemize}
\item   \noindent\underline{\sf Step 1:} Extrapolate the 
interface  $\Gamma_h^m$  with a second-order 
Adams-Bashforth method:
\[
\Gamma_h^m(\mathbf {x}^{m-1}) =
  \mathbf {x}^{m-1} +
  \delta t \bld \omega_\Gamma^{m-1/2},\quad \forall \mathbf{x}^{m-1}\in \Gamma_h^{m-1},
\]
where $\bld \omega_\Gamma^{m-1/2}$ is the mesh velocity 
on the interface that satisfies
\begin{align}
  \label{interface-vel}
\bld \omega_\Gamma^{m-1/2}\cdot\bld n|_{\Gamma_h^{m-1}} = (1.5\bld u^{m-1}-0.5\bld u^{m-2})\cdot\bld n|_{\Gamma_h^{m-1}}.
\end{align}
Note that  the tangential component of mesh velocity $\bld \omega_\Gamma^{m-1/2}$ can be {\it freely} chosen as it does not affect the shape of the interface.
%The advantage of this approach with the traditional (Lagrange) approach
%$\bld \omega_\Gamma= 1.5\bld u^{m-1}-0.5\bld u^{m-2}$ for the rising bubble benchmark test 
 \item\noindent  \underline{\sf Step 2:} Define the new mesh 
$\Th^{{m}}=\mathcal{A}_{h}^m(\Th^{m-1})$ by an extension operator:
\[
\mathcal{A}_{h}^m(\mathbf x^{m-1}) = \mathbf x^{m-1} +
      \mathsf{Ext}(\delta t\bld\omega_\Gamma^{m-1/2}),
\]
where $\mathsf{Ext}$ is a proper extension operator from the interface to 
the interior domain.
%which is taken to be 
%a psedoelasticity operator %\cite{Nobile01} 
%in the current  configuration $\Th^{m-1}$ for our numerical simulation.
 \item\noindent\underline{\sf Step 3:} Update the computation mesh $\Th^{m}$ and the interface $\Gamma_h^m$ and compute the approximatio to the mean curvature vector $\kappa\bld n|_{\Gamma_h^m}=-
\triangle_\Gamma(\mathbf x^m)|_{\Gamma_h^m}$ by the Laplace-Beltrami technique:
Find $\bld \kappa^m\in [V_{\Gamma_h^m}^k]^d
%H^1(\Gamma_h^m)
$ such that 
\begin{align}
  \label{curvature}
(\bld \kappa^m, \bld \psi)_{\Gamma_h^m} =(\nabla_\Gamma(\mathbf x^m), 
\nabla_{\Gamma}\bld \psi)|_{\Gamma_h^m},\quad\forall
\bld\psi\in[\mathcal{V}_{\Gamma_h^m}^k]^d,
\end{align}
where 
\begin{align}
  \label{vint}
  \mathcal{V}_{\Gamma_h^m}^k:=
\{\psi \in H^1(\Gamma_h^m):\quad 
              \psi|_{F}\in \mathcal{P}^{k}(F),
  \quad \forall F\in\Gamma_h^m 
\}.
\end{align}
Here $\nabla_\Gamma(\bullet) := \nabla(\bullet) -
(\nabla(\bullet)\cdot\bld n)\bld
n$ is the surface gradient, and $\triangle_\Gamma$ is the surface
Laplacian.
We also compute the mesh velocity $\bld \omega^m$ at time $t^m$ by BDF2:
\[
\bld \omega^m:= \frac{1}{\delta t}(1.5\mathcal{A}_h^m-2\mathcal{A}_h^{m-1}+
0.5\mathcal{A}_h^{m-2}).
\]
\item\noindent\underline{\sf Step 4:}
Update field unknowns at $t^m$ by the IMEX-ALE-TVNNS-HDG scheme \eqref{ale-hdg-imex} (with $s=2$), where the density $\rho$ and viscosity $\mu$ in equation \eqref{ale-hdg-imex1} are understood to be piecewise constant functions that need 
to be placed inside the operators, and the right hand side of \eqref{ale-hdg-imex1} has 
the following additional term due to surface tension:
\[
f_{\Gamma_h}(\bld\kappa_h^m, \bld v_h) := \sum_{F\in\Gamma_h^m}\int_F\tau\bld\kappa_h^m\cdot\bld v_h^-\,\mathrm{ds},
\] 
where $\bld v_h^-|_F = ((\bld v_h)|_{K^-})|_F$ with 
$K^-$ being the element sharing the facet $F$ that belongs to the interior subdomain $\Omega_2$.
For completeness, we present the resulting fully discrete formulation below:
Find  $(\bld u_h^m, p_h^m,
\widehat{\bld u}_h^m, \widehat{\sigma}_h^{nn,m})\in 
\bld 
V_h^{t_m}\times Q_h^{t_m}\times 
\widehat{\bld V}_{h,0}^{t_m}\times \widehat{M}_h^{t_m}$ 
  such that 
\begin{subequations}
  \label{ale-hdg-imex-tp}
  \begin{align}
  \label{ale-hdg-imex-tp1}
 \mathcal{M}_h(D_t^2 \bld u_h^m,   \rho\,\bld v_h)+
\mathcal{C}_h^{dg}\left(\bld \omega^m,\widetilde{\bld u}_h^{m,2}, \rho\,\bld
  v_h\right)\hspace{16ex}&\\
+
  \mathcal{B}_h\left((\bld u_h^m,  
\widehat{\bld u}_h^m), 
    ( 2\mu\,\bld v_h, 
    2\mu\, \widehat{\bld v}_h)\right) 
    -
    \mathcal{D}_h\left(\bld v_h, 
    (p_h^m, 
  \widehat{\sigma}^{nn,m}_h)\right)
    &=\; {f}_h(\bld v_h)+f_{\Gamma_h}(\bld \kappa_h^m,\bld v_h),\nonumber\\
  \label{ale-hdg-imex-tp2}
    \mathcal{D}_h\left(\bld u_h^m, 
    (q_h, 
  \widehat{\tau}^{nn}_h)\right) &=\; 0,
  \end{align}
\end{subequations}
for all 
$(\bld v_h, q_h,
\widehat{\bld v}_h, \widehat{\tau}_h^{nn})\in 
\bld 
V_h^{t_m}\times Q_h^{t_m}\times 
\widehat{\bld V}_{h,0}^{t_m}\times \widehat{M}_h^{t_m}$.
%, where the domain 
%velocity $\bld \omega$ is approximated by 
%\[
%  \bld\omega =\frac{ \mathcal{A}_{t^{n+1}}-\mathcal{A}_{t^{n}}}{\delta t}\circ
%  \mathcal{A}_{t^{n+1}}^{-1}(\mathbf x).
%\]
%\end{minipage}
\end{itemize}

\newcommand{\vgamma}{\mathcal{V}_{\Gamma_h^{m-1}}^k}
\newcommand{\vale}{\mathcal{V}_{h}^{m-1, k}}

\begin{remark}[Step 1: Interface node distribution]
  \label{rk:step1}
  The choice of $\bld \omega_\Gamma^{m-1/2}$ strongly affect the node distributions 
on the interface.
Traditionally,  $\bld \omega_\Gamma^{m-1/2}$ is chosen to be, either the full fluid
velocity 
\begin{align*}
  \bld \omega_{\Gamma}^{m-1/2} = 1.5\bld u^{m-1}-0.5\bld u^{m-2}
,\end{align*}
or the normal component of the fluid velocity 
\begin{align}
  \label{omega-u}
  \bld \omega_{\Gamma}^{m-1/2} = ((1.5\bld u^{m-1}-0.5\bld u^{m-2})\cdot\bld
n)\bld n.
\end{align}
However, both approaches suffer from the general problem that there is no
control of the node distribution and of the shape of surface elements.
We follow the work \cite{Barrett07}, see also \cite{Ganesan17},
to update the mesh velocity that combines the equation \eqref{interface-vel} 
with the Laplace-Beltrami identity 
$-\triangle_{\Gamma}\mathsf{id} = \kappa\bld
n$ for the curvature
 in order to get more control over the node
distribution.
In particular, we find
$(\bld \omega_{\Gamma}^{m-1/2}, \kappa_h)\in [\vgamma]^d\times\vgamma$
such that
\begin{subequations}
  \label{omega-eq}
\begin{align}
  \langle \bld \omega_{\Gamma}^{m-1/2}\cdot\bld n,
\phi
  \rangle_{\Gamma_h^{m-1}} = &\;
  \langle (1.5\bld u^{m-1}-0.5\bld u^{m-2})\cdot\bld n,
\phi
  \rangle_{\Gamma_h^{m-1}},\\
  \langle\kappa,\bld\psi\cdot\bld n 
  \rangle_{\Gamma_h^{m-1}}
  - \delta t 
  \langle\nabla_{\Gamma}\bld \omega_{\Gamma}^{m-1/2}, 
  \nabla_{\Gamma}\bld \psi
  \rangle_{\Gamma_h^{m-1}}
  = &\;
 \langle\nabla_{\Gamma}\mathbf x^{m-1}, 
  \nabla_{\Gamma}\bld \psi
  \rangle_{\Gamma_h^{m-1}},
\end{align}
\end{subequations}
for all
$(\bld \psi, \phi)\in [\vgamma]^d\times\vgamma$.
Numerical comparion of the approaches \eqref{omega-u} and \eqref{omega-eq}
will be presented in Section \ref{sec:num} for the rising bubble benchmark problem \cite{Hysing09}, which show the superior performance of approach \eqref{omega-eq}
in terms of interface node distribution.
\end{remark}

\begin{remark}[Step 2: Construction of the ALE map]
  \label{rk:step2}
%  The extension operator $\mathsf{Ext}$ controls quality of the 
%  mapped mesh. 
Various procedures were introduced in the literature to extend interface
deformation to the interior domain, see, e.g., 
\cite{Yirgit08, Wick11,Persson15, Basting17}.
In this paper, we construct the ALE map
$\mathcal{A}_h^m\in [\vale]^d$
using harmonic extension with
stiffening \cite{Richter17} on the updated mesh $\Th^{m-1}$ as follows:
Find 
$\mathcal{A}_h^m\in [\vale]^d$ with 
$\mathcal{A}_h^m|_{\Gamma_h^{m-1}}=
\mathbf x^{m-1}+\delta
t\bld\omega_{\Gamma}^{m-1/2}$
such that 
\begin{align*}
  (\alpha \nabla \mathcal{A}_h^m, \nabla \phi)_{\Th^{m-1}} = 0,
  \quad \forall \bld \phi\in [\vale]^d \text{ with } \bld
  \phi|_{\Gamma_h^{m-1}} = 0,
\end{align*}
where the stiffening coefficient $\alpha: \Th^{m-1}\rightarrow \mathbb{R}$ 
takes maximum value $10$ on the interface and decreases rapidly to $1$
away from the interface based on the distance of the point to the nearest interface point, and 
the space 
\[
  \vale:=
    \{\psi \in H_0^1(\Th^{m-1}):\quad 
              \psi|_{T}\in \mathcal{P}^{k}(T),
              \quad \forall T\in\Th^{m-1}
\}.
\]
\end{remark}

\begin{remark}[Step 3: Mean curvature vector approximation]
  \label{rk:step3}
  We use an $L^2$-project \eqref{curvature} with isoparametric finite
  elements of degree $k$ to approximate the mean curvature vector. 
  It was shown in \cite{Heine04} that the order of convergence in 
  the $L^2$-norm for the approximation is only $k-1$.
  This will negatively impact  the accuracy of the 
  velocity approximation in Step 4, which in turn negatively 
  affects the accuracy of
  the interface approximation in Step 1.
The construction of more accurate mean curvature vector approximation
consists of our ongoing work.
We mention that mean curvature vector 
approximation with first order $L^2$-convergence can be obtained 
for linear finite
elements with stabilization on piecewise linear surface \cite{Hansbo15}.
\end{remark}

\begin{remark}[Step 4: Consistency]
  \label{rk:step4}
We show that  the equation \eqref{ale-hdg-imex-tp1} is a consistent discretization of
  the momentum equation \eqref{two-eq-ale1} with 
  the interface condition \eqref{two-eq-ale4}.
  Let $\bld u(\mathbf x, t)$ and $p(\mathbf x, t)$ be the smooth solution 
  to the equations \eqref{two-phase-eq}. 
  For simplicity, we only consider the semi-discrete case (continuous in
  time).
  Assume 
  \[\widehat{\sigma}^{nn}|_F:= \left\{
      \begin{tabular}{ll}
        $ (\bld \sigma\bld n)\cdot\bld n$ & if $F\not\in \Gamma_h^m$,\\[2ex]
        $ (\bld \sigma_1\bld n)\cdot\bld n$ & if $F\in \Gamma_h^m$,
\end{tabular}\right.
  \]
  Then, by intergration by parts, we have 
  \begin{align*}
&\; 
\mathcal{M}_h(\left.\frac{\partial\bld u}{\partial t}\right|_{\mathbf x_0}, \rho\,\bld
v_h)+
\mathcal{C}_h^{dg}\left(\bld \omega,{\bld u}, \rho\,\bld
v_h\right)+
  \mathcal{B}_h\left((\bld u,  
{\bld u}), 
    ( 2\mu\,\bld v_h, 
    2\mu\, \widehat{\bld v}_h)\right) 
    -
    \mathcal{D}_h\left(\bld v_h, 
    (p, 
  \widehat{\sigma}^{nn})\right)
\\
&
=\;\left(
\rho(\left.\frac{\partial\bld u}{\partial t}\right|_{\mathbf x_0}
+(\bld u-\bld \omega)\cdot\nabla\bld u)
-\mathrm{div}_{\mathbf x}\bld \sigma,\bld
v_h
\right)_{\Th^m}
-\langle(\bld \sigma_2-\bld\sigma_1)\bld n_2,\bld v_h^-
\rangle_{\Gamma_h^m}\\
&=\;
(\rho\bld f,\bld
v_h)_{\Th^m}
+\langle\tau\kappa\bld n_2,\bld v_h^-
\rangle_{\Gamma_h^m}\\
&=\;
f_h(\bld v_h)
+
f_{\Gamma_h}(\kappa\bld n_2, \bld v_h),
\end{align*}
where $\bld n_2$ is the normal vector pointing from $\Omega_2$ to $\Omega_1$.
Hence, the spatial discretization in equation \eqref{ale-hdg-imex-tp1} is
consistent with the equations \eqref{two-eq-ale1}--\eqref{two-eq-ale4}.
\end{remark}

\section{Numerical results}
\label{sec:num}
Here we show three numerical studies for the  
IMEX-ALE-TVNNS-HDG schemes \eqref{ale-hdg-imex} and 
\eqref{ale-hdg-imex-tp}. 
The first example focus on the accuracy study
of the proposed method  with a prescribed smooth 
ALE map.  
We observe optimal convergence.
The second example deals with the application of our method to
a free boundary problem, where the ALE map is not a priori given.
The last example solves the classical rising bubble benchmark problem 
\cite{Hysing09}
in two-phase flow.
Our numerical simulations are performed using the open-source finite-element software 
{\sf NGSolve} \cite{Schoberl16}, \url{https://ngsolve.org/}.

\newcommand{\ya}{x_{0,1}}
\newcommand{\yb}{x_{0,2}}
\newcommand{\xa}{x_1}
\newcommand{\xb}{x_2}
\subsection{Accuracy test.}
We consider the  Navier-Stokes equations \eqref{ns-eq-ale} 
on the domain $\Omega^t=[0,1]\times [0,1]$ for $t\in [0,0.5\pi]$ with
homogeneous Dirichlet boundary conditions for velocity, and 
choose the source term such that the exact
solution is given as follows:
\begin{alignat*}{2}
  \bigg\{\begin{split}
    u_1(\mathbf x, t) &= \;
  4(\xa(1-\xa))^2(2\xb-6\xb^2+4\xb^3)\sin(t),\\
    u_2(\mathbf x, t) &=\; 
  -4(\xb(1-\xb))^2(2\xa-6\xa^2+4\xa^3)\sin(t),
  \end{split}
  &&\quad\quad p(\mathbf x, t) =\; \sin(\xa+\xb).
  \end{alignat*}
\noindent Density $\rho$ is taken to be $1$ and we 
consider viscosity $\mu$ either to be $1$ or $10^{-6}$.
We use the following prescribed ALE map 
$
  \mathbf x(\mathbf x_0, t) =  
  \mathcal{A}_t(\ya, \yb) = 
\left(\ya + u_1(\mathbf x_0, 2t),\;\; 
\yb + u_2(\mathbf x_0, 2t)\right)
$
for the mesh movement, where $\bld u = (u_1, u_2)$,
$\mathbf x = (\xa, \xb),$ and $\mathbf x_0 = (\ya, \yb)$.
%The x-component of velocity contour on a sample mesh  
%with mesh size $h=1/16$ is shown in Fig.~\ref{fig:ac} for 
%$t=.25\pi$ and $t=.5\pi$.
%\begin{figure}[h!]
%\centering
%\includegraphics[width=.3\textwidth]{nsac1}\quad\quad\quad
%\includegraphics[width=.3\textwidth]{nsac2}
%  \caption{Sample mesh and x-component velocity contour. Left: $t=.25\pi$.
%  Right: $t=.5\pi$.}
%\label{fig:ac}
%\end{figure} 

We compare our HDG method \eqref{ale-hdg-imex}, denoted as \HDG, with 
the ALE method using Scott-Volegius element (on barycentric refined meshes), 
denoted as \SV, and that using high-order Taylor-Hood element,
denoted as \TaH. We use either polynomials of degree $k=2$ or $k=3$ for 
the velocity approximation. % for all three methods.
For the time discretization, 
we use the third order IMEX-SBDF3 method \cite{AscherRuuthWetton93}. 
Time step size is taken to be small enough so that spatial error dominates
temporal error.
%Initial conditions are taken to
%be the interpolation of the true solution at $0, \delta t, 2\delta t$.
%Table~\ref{table:errp2} and 
Table~\ref{table:errp3}
lists the history of
convergence for the  $L^2$-error in
the velocity approximation
%with viscosity $\mu=1$ and $\mu=10^{-6}$ 
for  $k=2$ and $k=3$. % respectively. 
We observe that 
the order of convergence for the error for all three methods are optimal
($k+1$) for the case $\mu=1$. But for 
the convection-dominated case $\mu=10^{-6}$, the accuracy for \HDG\;
clearly outperforms those for \SV\; and \TaH.
In particular, \HDG\; velocity 
error is not affected by the viscosity coefficient,
however, \SV\; looses roughly one order of approximation  and \TaH\; 
looses up to two order of approximation for the case $\mu=10^{-6}$.
%for both loose roughly one order of approximation. We also observe
%that \TaH\; without grad-div stabilization produces even less accurate results
%for $\mu=10^{-6}$, with only second order of convergence observed for the case
%$k=3$. 
\begin{table}[ht!]
\begin{center}
% \footnotesize
\scalebox{0.78}{%
\begin{tabular}{| c|c | c | c | c | c | c | c | c 
  | c | c | c | c | c | c |c|c|c|c|c|} \hline
  &   &\multicolumn{4}{c|}{\HDG} & 
  \multicolumn{4}{c|}{\SV} 
  & \multicolumn{4}{c|}{\TaH} 
  \\\hline
  &   &
  \multicolumn{2}{c|}{$\mu=1$} & \multicolumn{2}{c|}{$\mu=10^{-6}$}& 
  \multicolumn{2}{c|}{$\mu=1$} & \multicolumn{2}{c|}{$\mu=10^{-6}$}& 
  \multicolumn{2}{c|}{$\mu=1$} & \multicolumn{2}{c|}{$\mu=10^{-6}$} 
  \\\hline
  $k$& $1/h$ & Error & Order  & Error & Order  & Error & Order& Error & Order& Error & Order& Error & Order\\
\hline 
 &8  & 2.27e-04  &  --   & 1.45e-04  &  --  & 4.74e-04  &  --   & 4.93e-04  &  --   & 1.71e-04  &  --   & 7.08e-03  &  --  \\ 
2&16 & 2.24e-05  &  3.34 & 1.88e-05  &  2.95& 5.49e-05  &  3.11 & 1.11e-04  &  2.15 & 2.12e-05  &  3.01 & 2.04e-03  &  1.80\\ 
 &32 & 2.46e-06  &  3.19 & 2.37e-06  &  2.99& 6.30e-06  &  3.12 & 2.60e-05  &  2.10 & 2.65e-06  &  3.00 & 5.19e-04  &  1.97\\ 
\hline 
  & 8  & 1.39e-05  &  --   & 1.13e-05  &  --     & 1.88e-05 &  --  & 3.02e-05  &  --  & 1.05e-05  &  --   & 6.62e-04  &  --  \\   
  3& 16 & 7.68e-07  &  4.17 & 7.23e-07  &  3.96  &1.06e-06 &  4.15& 2.78e-06  &  3.44& 6.41e-07  &  4.04 & 1.76e-04  &  1.91\\ 
  & 32 & 4.46e-08  &  4.11 & 4.57e-08  &  3.99 &6.15e-08 &  4.10& 2.90e-07  &  3.26& 3.93e-08  &  4.03 & 4.25e-05  &  2.05\\ 
\hline 
\end{tabular}
  }
\caption{\it History of convergence of the $L^2$-error for the  velocity approximation.}
\label{table:errp3}
\end{center}
\end{table}
%\newpage % FIXME

\subsection{Free boundary problem.}
Here we apply our HDG method \eqref{ale-hdg-imex} to a free boundary problem 
\cite{Ramaswamy87,Duarte04} that models the propagation of a solitary wave.
The ALE Navier-Stokes equations \eqref{ns-eq-ale} is solved on the
free-boundary domain shown in Fig.~\ref{fig:free}.
%\begin{wrapfigure}{r}{.45\textwidth}
\begin{figure}[ht!]
%  \vspace{-4ex}
  \begin{center}
\includegraphics[width=.45\textwidth]{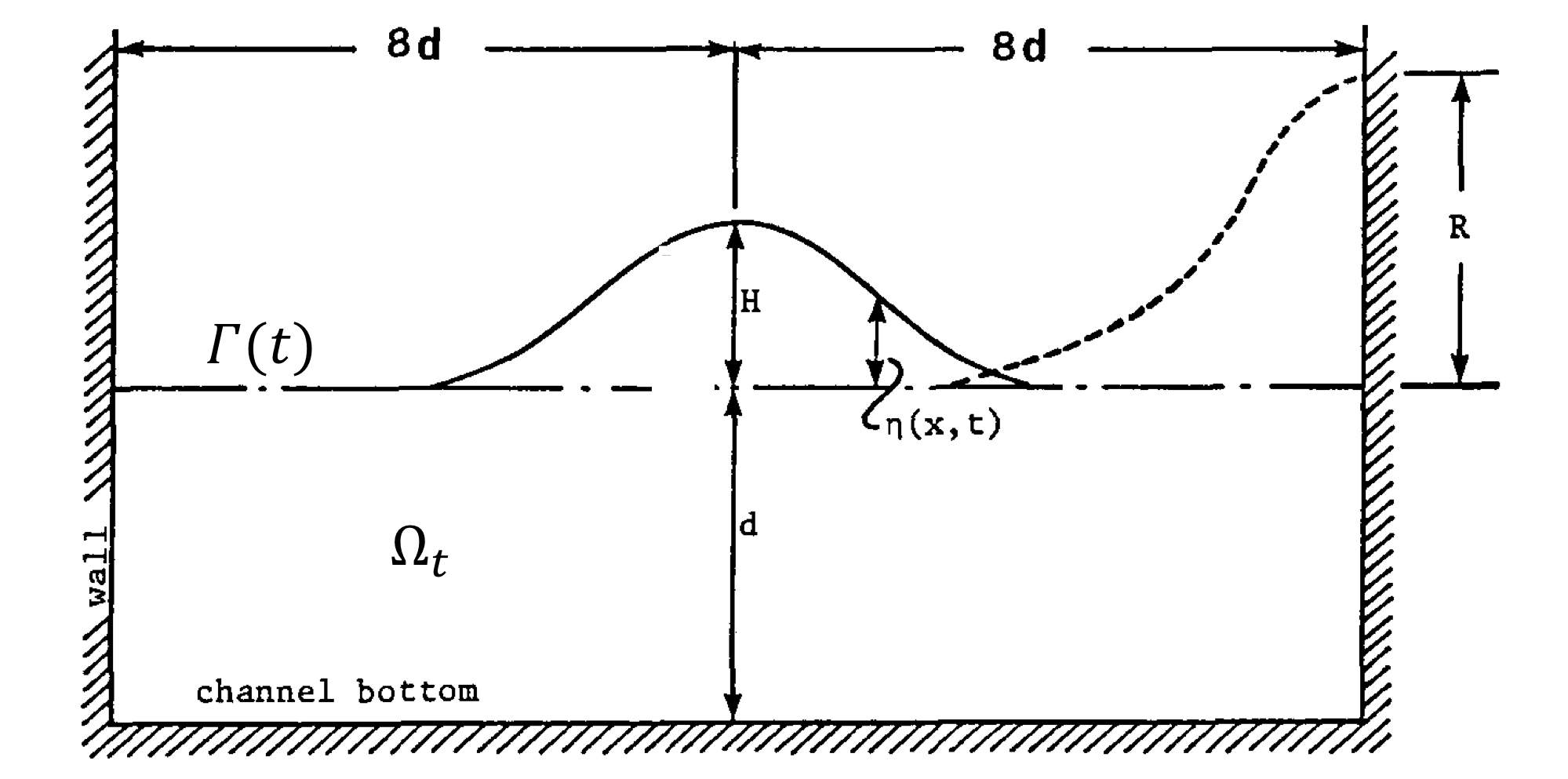}
  \end{center}
%  \vspace{-5ex}
    \caption{\it The free-boundary domain.}
\label{fig:free}
%  \vspace{-2ex}
%\end{wrapfigure} 
\end{figure} 
The top boundary $\Gamma(t)$ in Fig.~\ref{fig:free} is a free surface on which 
we impose a stress-free boundary condition.
Free-slip boundary condition is used for the other boundaries. 
The domain velocity on the free surface is set to be equal to the fluid
velocity, 
\begin{align}
  \label{free-bc}
  \left.\frac{\partial \mathbf x}{\partial
  t}\right|_{\mathbf x_0} = 
  \bld \omega(\mathbf x, t) = \bld u(\mathbf x, t), \quad \forall 
  \mathbf x\in \Gamma(t),
\end{align}
which expresses the fact that the free boundary $\Gamma(t)$ consists
of the same particles for all $t>0$.
The initial condition is taken from Laitone's solitary wave approximation
\cite{Laitone60}, with the free surface elevation $\eta$ given by
%\[
$
\eta = d+H\mathrm{sech}^2\left(\sqrt{\frac{3H}{4d^3}}\mathbf x_1\right),
%\]
$
and $H=2$ is the initial wave height and $d=10$ is the still water depth.
Density and viscosity are taken to be $1$, and the source $\bld f$ is 
a gravitational acceleration of magnitude $9.8$ acts vertically on the downward
direction.

For this problem the ALE map is not prescribed. We apply the semi-implicit
approach \cite{Fernandez07} to update the domain, see 
Section \ref{sec:twophase}. 
We use IMEX-SBDF2 time stepping ($s=2$) with 
the HDG scheme \eqref{ale-hdg-imex} using quadratic velocity approximation on a mesh with 
mesh size $h=2$. Time step size is taken to be $\delta t = 0.025$, and the final
time $T=12$.
A simple harmonic extension on the reference configuration is adopted for
the ALE map, which is approximated by continuous linear finite elements.
Fig.~\ref{fig:free-soln} presents the velocity
magnitude at some subsequent times along with the mesh.
\begin{figure}[ht!]
  \begin{center}
\includegraphics[width=.8\textwidth]{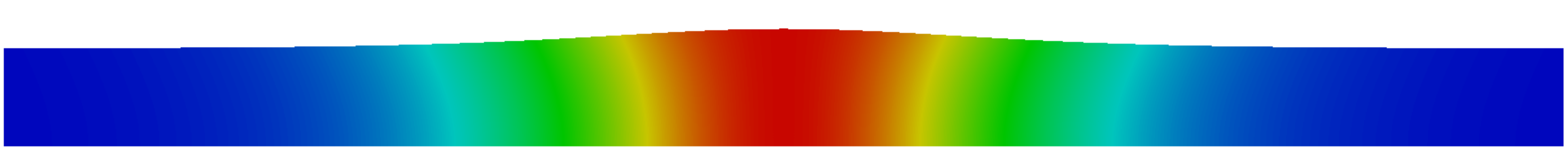}
\includegraphics[width=.8\textwidth]{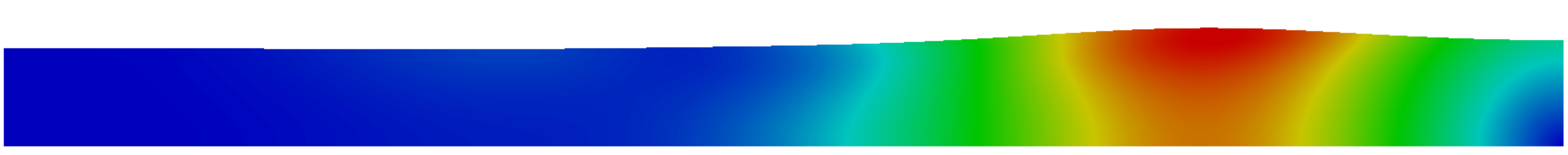}
\includegraphics[width=.8\textwidth]{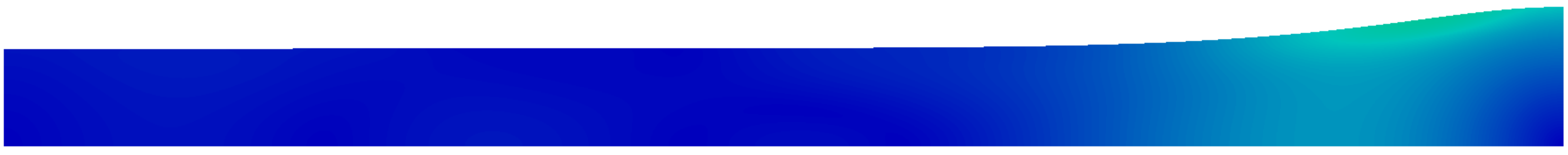}
\includegraphics[width=.8\textwidth]{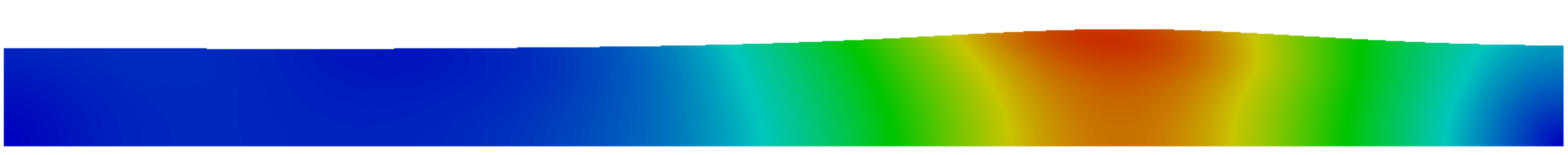}
  \end{center}
    \caption{\it Velocity magnitude (from left to right and top to
    bottom) at time 
  $t=0, 4, 8, $ and $12$.
  }
\label{fig:free-soln}
\end{figure} 
Our numerical simulation agrees well with computational results given by 
\cite{Ramaswamy87,Duarte04}. 
See Table~\ref{table:free} for a comparison of 
the run-up height, the time when the wave hits the
right side wall and the maximum pressure.
See also Fig.~\ref{fig:free-height} for the plot of the run-up height versus
time, which matches well with numerical results reported in \cite{Duarte04}.

%\begin{wraptable}{l}{.5\textwidth}
\begin{table}[ht!]
\begin{center}
  %\footnotesize
  \scalebox{1}
  {
  \begin{tabular}{| c | c | c | c |} \hline
  & Duarte\cite{Duarte04} &Ramaswamy\cite{Ramaswamy87}  &{\HDG} 
  \\\hline
  Height & 14.27 &14.48& 14.33\\
  Time & 7.7 & 7.6& 7.65\\
  Pressure & 130 & 131.66 & 131.89\\
\hline 
\end{tabular}
}
\end{center}
  \caption{\it Comparison table: height, time and pressure.}
\label{table:free}
%\end{wraptable}
\end{table}

%\begin{wrapfigure}{r}{.45\textwidth}
\begin{figure}[ht!]
  \begin{center}
 \begin{tikzpicture} % run-up height 
          \tikzstyle{every node}=[font=\footnotesize]
\begin{axis}[
	width=0.4\textwidth,
	height=0.2\textheight,
 	xlabel={time},
	ylabel={Maximum Height},
    xtick={0,2,4,6,8,10,12},
    ytick={11.5, 12, 12.5, 13,13.5, 14, 14.5},
    xmin=0,
    xmax=12,
    ymin=11.5,
    ymax=14.5,
   	yticklabel style={/pgf/number format/fixed,/pgf/number format/precision=1},   	
   	every axis plot/.append style={line width=0.8pt, smooth},
   	no markers,   	
   	legend style={at={(axis description cs:0.01,0.65)},anchor=south west,
    draw=none},
  x label style={at={(axis description cs:0.5, 0.05)},anchor=north},
  ]
\addplot[black] table[x index=0, y index=1]{soli.dat};
\addplot[black, dashed] table[x index=0, y index=1]{soliref.dat};
\addlegendentry{\HDG}
\addlegendentry{Duarte \cite{Duarte04}}
\end{axis}
\end{tikzpicture}
\end{center}
  \vspace{-4ex}
  \caption{\it Run-up height.}%: \HDG\;(solid), Duarte \cite{Duarte04} (dashed).}
  \label{fig:free-height}
%\end{wrapfigure}
\end{figure}
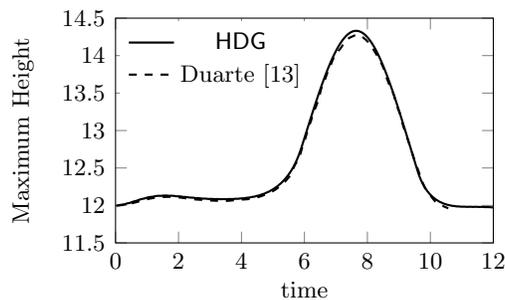

\subsection{Rising bubble problem}
Here we solve the rising bubble problem \cite{Hysing09} using 
the IMEX-ALE-TVNNS-HDG method detailed in Section \ref{sec:twophase}.
The initial configuration, see Fig.
\ref{fig:bubble}, consists of a circular bubble of radius 
$r_0=0.25$ centered at $(0.5,0.5)$ in a $[1\times 2]$ rectangular domain.
\begin{figure}[h!]
\centering
\includegraphics[width=.3\textwidth]{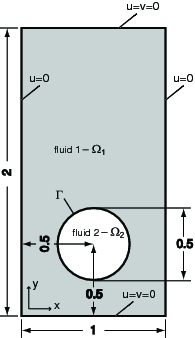}
  \caption{Initial configuration and boundary conditions for the rising bubble
  problem.}
\label{fig:bubble}
\end{figure}
The density of the bubble is smaller than that of the surrounding fluid 
$(\rho_2<\rho_1)$. The no-slip boundary condition $(\bld u = 0)$ is used at
the top and bottom boundaries, whereas the free slip condition ($\bld
u\cdot{\bld{n}}=0$, 
$(\bld\sigma\bld n)\cdot\bld t=0$) is imposed on the vertical walls.

The same two test cases 
%with small and large density ratios 
proposed in \cite{Hysing09} are considered in this study.
Table \ref{table:bubble} 
lists the fluid and physical parameters use for the two test cases. 
%\begin{wraptable}{l}{.5\textwidth}
\begin{table}[ht!]
\begin{center}
  %\footnotesize
  %\scalebox{1}
  {
  \begin{tabular}{ c  c  c  c c c c c} \hline
    Test Case & $\rho_1$ &$\rho_2$& $\mu_1$&$\mu_2$&$\bld f$
              &$\tau$\\
    1 &	1000 &	100 	&10 	&1 &	(0,-0.98)& 	24.5\\
    2 &	1000 &	1 	&10 	&0.1 &	(0,-0.98)& 	1.96\\
    \hline 
\end{tabular}
}
\end{center}
\caption{\it Physical parameters 
  %and final time 
for the test cases.}
\label{table:bubble}
%\end{wraptable}
\end{table}
We take final time for Test Case 1 to be $T=3$, and 
for Test Case 2 to be $T=2$.
We note that for Test Case 1, the surface tension effects are strong enough
to hold the bubble together, which will end up in a
ellipsoidal shape,
but for Test Case 2, the decrease in surface tension causes 
the bubble to assume a more complex shape and develop thin filaments which
eventually break off. The time of break up for Test Case 2 is 
predicted \cite{Hysing09} to occur between 
$t=2.2$ and $2.4$.
The current ALE scheme without remeshing can not treat interface breakup,
so we stop the simulation at time $T=2$ for this case.

For both cases, we consider the IMEX-SBDF2 based ALE-TVNNS-HDG scheme in Section
\ref{sec:twophase} with polynomial degree $k=2$ and
quadratic isoparametric triangular elements
on 
three different  meshes as depicted in Fig.~\ref{fig:bm}.
\begin{figure}[h!]
\centering
\includegraphics[width=.28\textwidth]{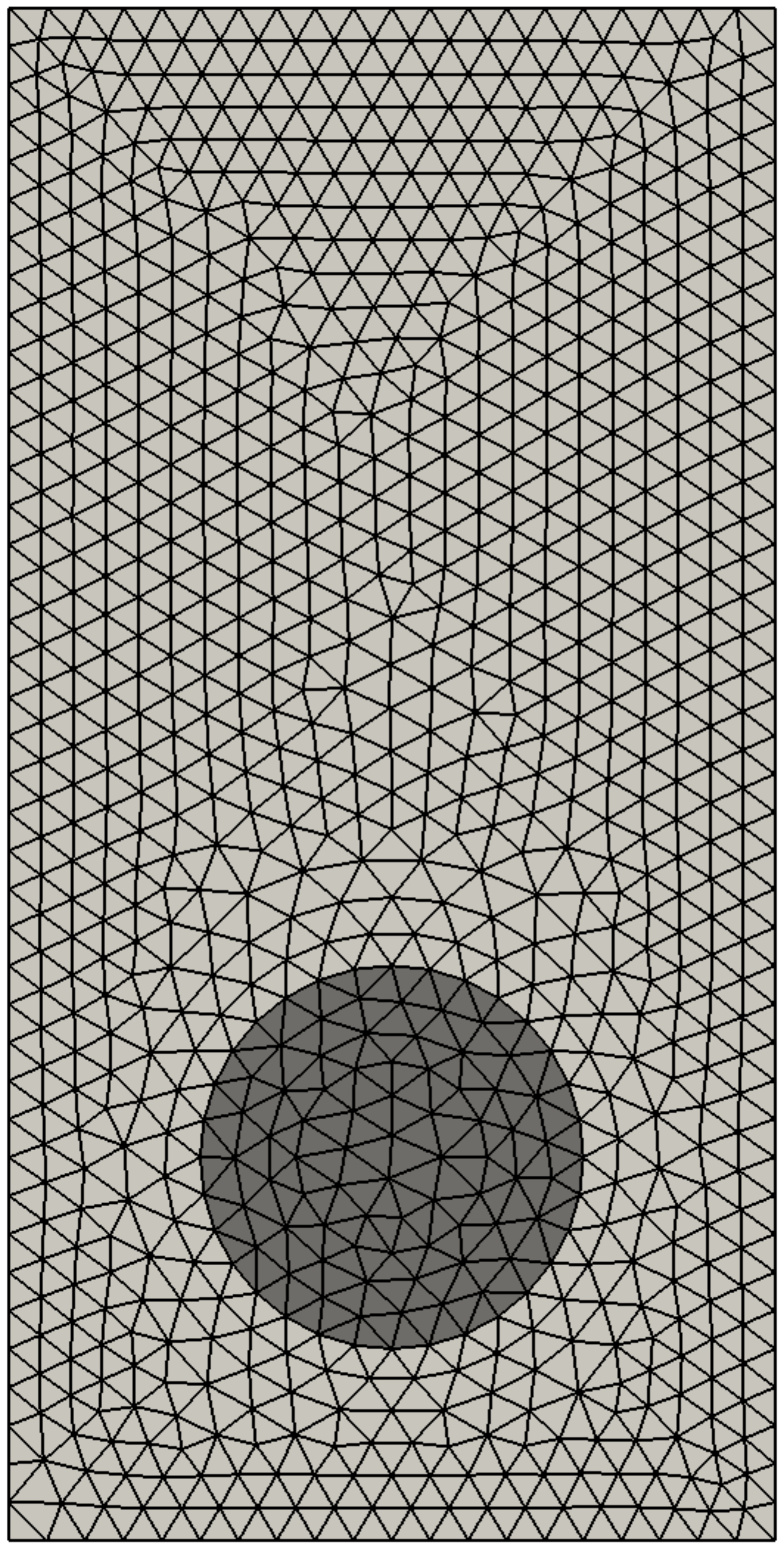}
\includegraphics[width=.28\textwidth]{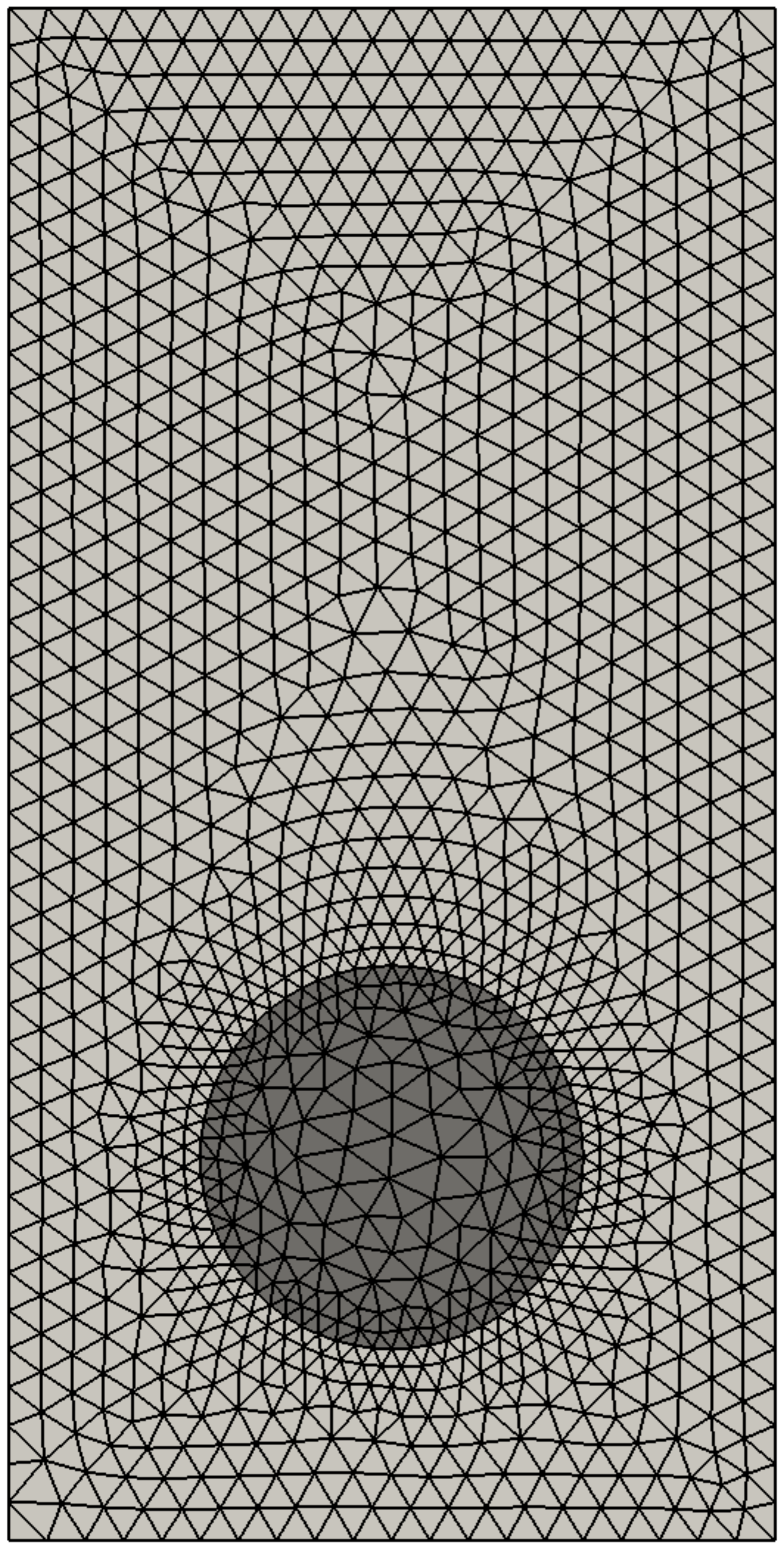}
\includegraphics[width=.28\textwidth]{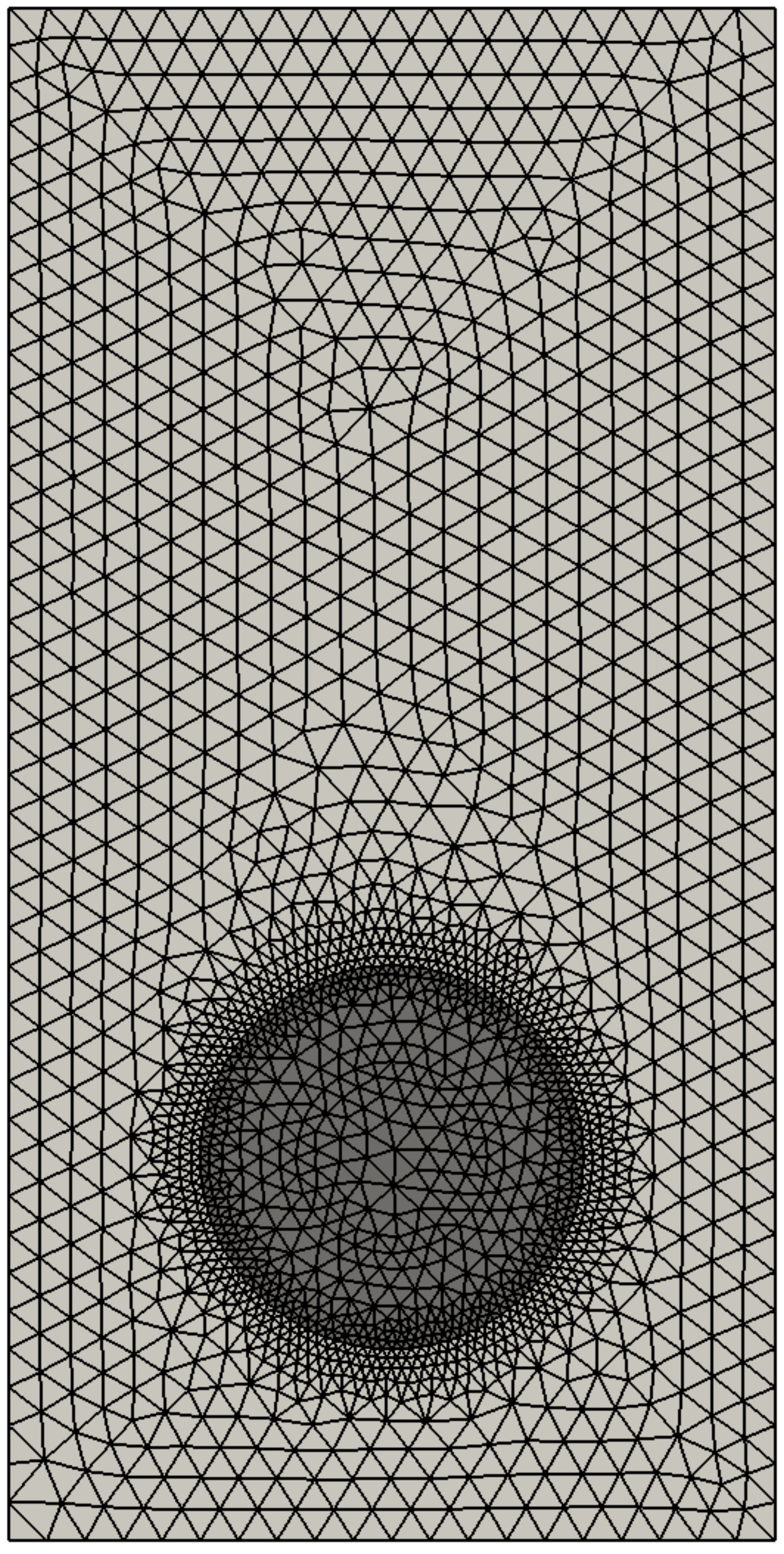}
  \caption{Left: mesh {\sf M1}.
    Middle: mesh {\sf M2}.
  Right: mesh {\sf M3}.
Dark color: domain $\Omega_2$.}
\label{fig:bm}
\end{figure}
Mesh {\sf M1} has mesh size $h=1/20$ away from the interface and 
$h_{\Gamma}=1/20$ near the interface, which 
consists of 941 nodes and 1260 elements.
Mesh {\sf M2} has mesh size $h=1/20$ away from the interface and 
$h_{\Gamma} = 1/40$ near the interface, which 
consists of 1216 nodes and 2310 elements.
Mesh {\sf M3} has mesh size $h=1/20$ away from the interface and 
$h_{\Gamma} = 1/80$ near the interface, which 
consists of 1216 nodes and 2310 elements.
For all the simulation, we take time step size to be 
$\delta t = h_{\Gamma}/8$ for Test Case 1 and $\delta t = h_{\Gamma}/4$
for Test Case 2.

\subsubsection{Interface node distribution}
We first compare the mesh quality for two different choices of 
interface velocity updates mentioned in Remark \ref{rk:step1}.
We consider Test Case 1 and use the coarse mesh {\sf M1}.
%While the scheme based on velocity update \eqref{omega-eq} produces 
%good quality mesh throughout the simulation till time $T=3$, 
%the scheme based on velocity update \eqref{omega-u} crashes at around
%$t=1.65$. 
The mesh deformation for the approaches 
\eqref{omega-u} and \eqref{omega-eq} at time 
$t=1.5$ and for the approach \eqref{omega-eq} at time $t=3$
are shown in Fig.\ref{fig:bm2}.
It is clear to observe that for the approach \eqref{omega-u}, the 
interface nodes tends to accumulate at the bottom of the bubble, which 
eventually leads to mesh tangling and crash of the simulation at around $t=1.65$.
However, for the approach \eqref{omega-eq}, the interface nodes seem to be
equidistributed, which produce good quality mesh throughout the simulation.
For this reason, we choose to use the approach \eqref{omega-eq}
for the interface velocity update for all the tests.

\begin{figure}[h!]
\centering
\includegraphics[width=.28\textwidth]{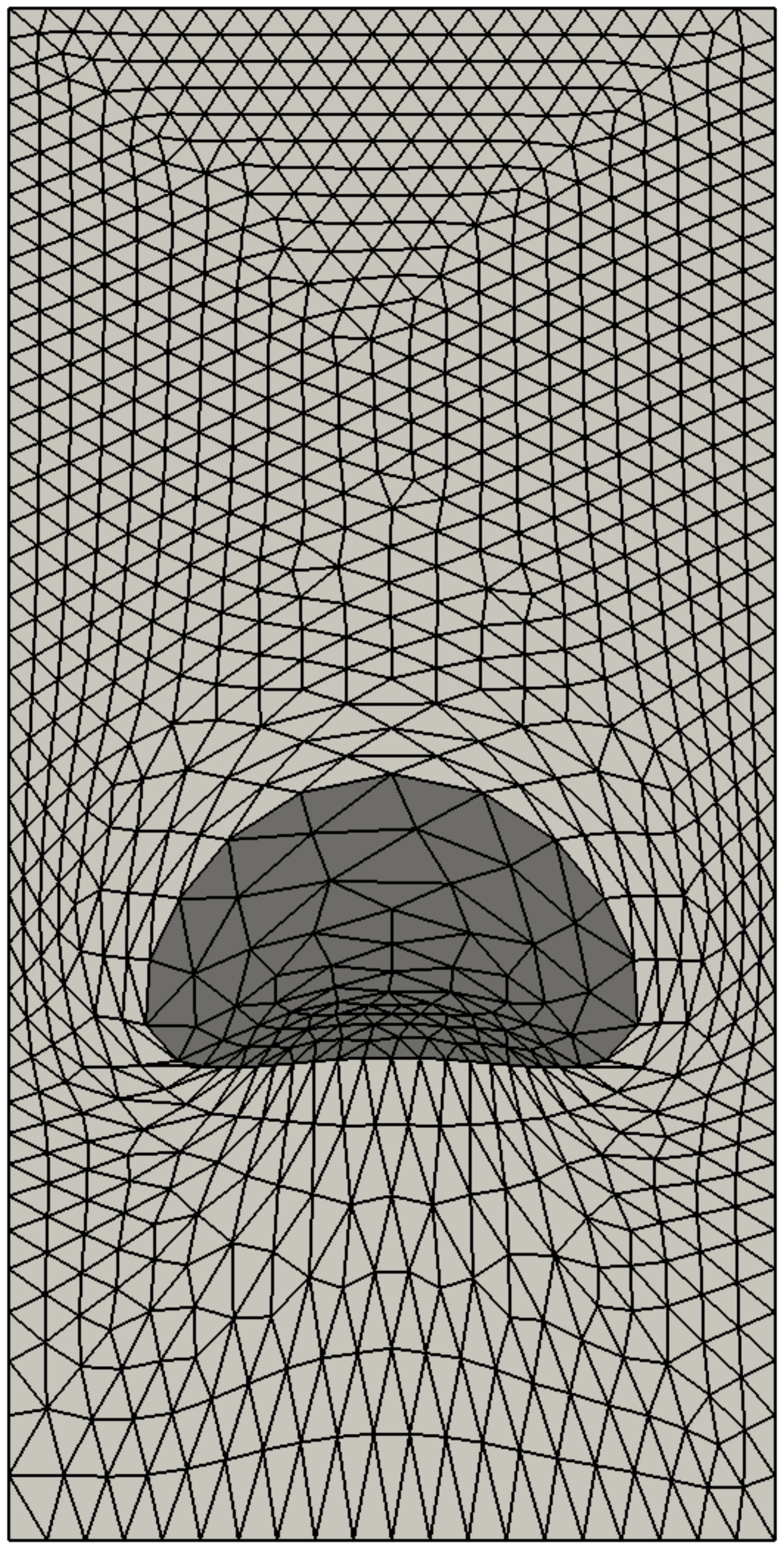}
\includegraphics[width=.28\textwidth]{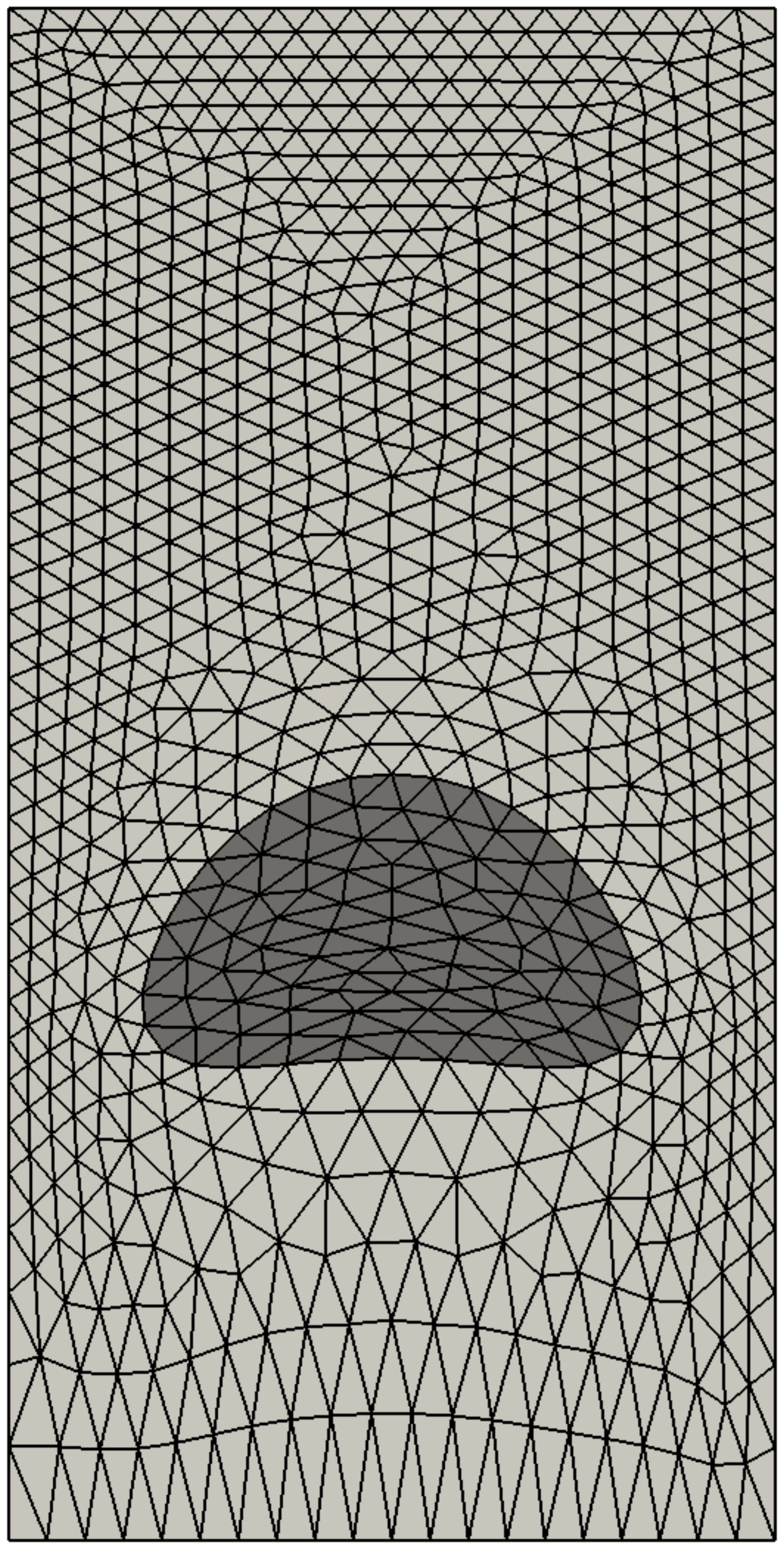}
\includegraphics[width=.28\textwidth]{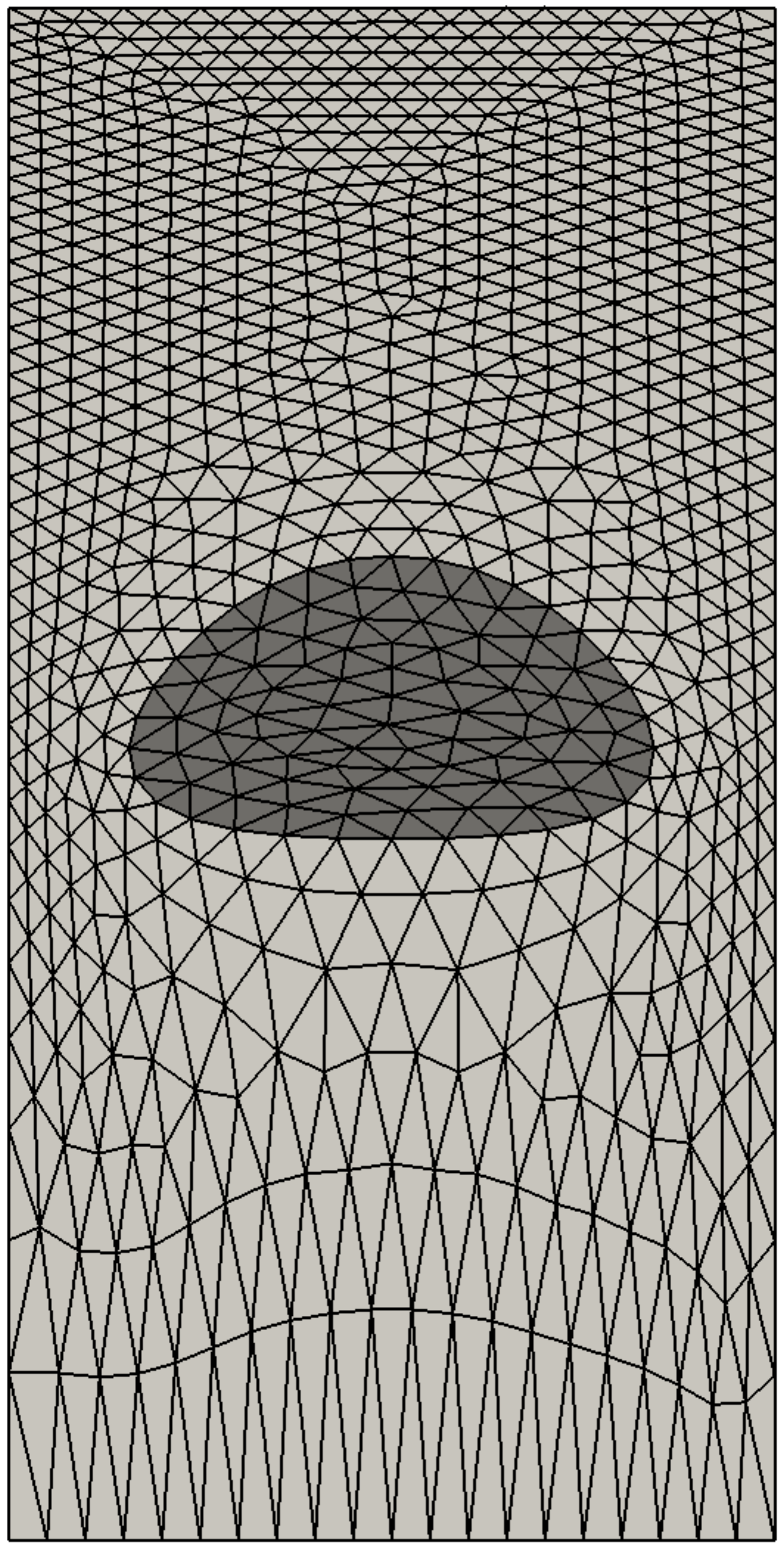}
  \caption{Mesh deformation on mesh {\sf M1} for Test Case 1
    with different interface velocity update approaches.
  Left: approach \eqref{omega-u} at $t=1.5$. 
Middle: approach \eqref{omega-eq} at $t=1.5$.
Right: approach \eqref{omega-eq} at $t=3$.}
\label{fig:bm2}
\end{figure}

\subsubsection{Bubble shape}
The bubble shapes at final time are present in 
Fig.~\ref{fig:bs}.
For Test Case 1, we also present the reference data
from \cite{Hysing09} using the MooNMD code 
on the finest mesh 
with $900$ degrees of freedom on the interface
and $6000$ total time steps.
We observe that for Test Case 1, 
the present results are indistinguishable
from the reference data even for the coarsest mesh {\it M1}.
We also observe similar results 
for Test Case 2 on all three meshes. 
\begin{figure}[h!]
\centering
\includegraphics[width=.4\textwidth]{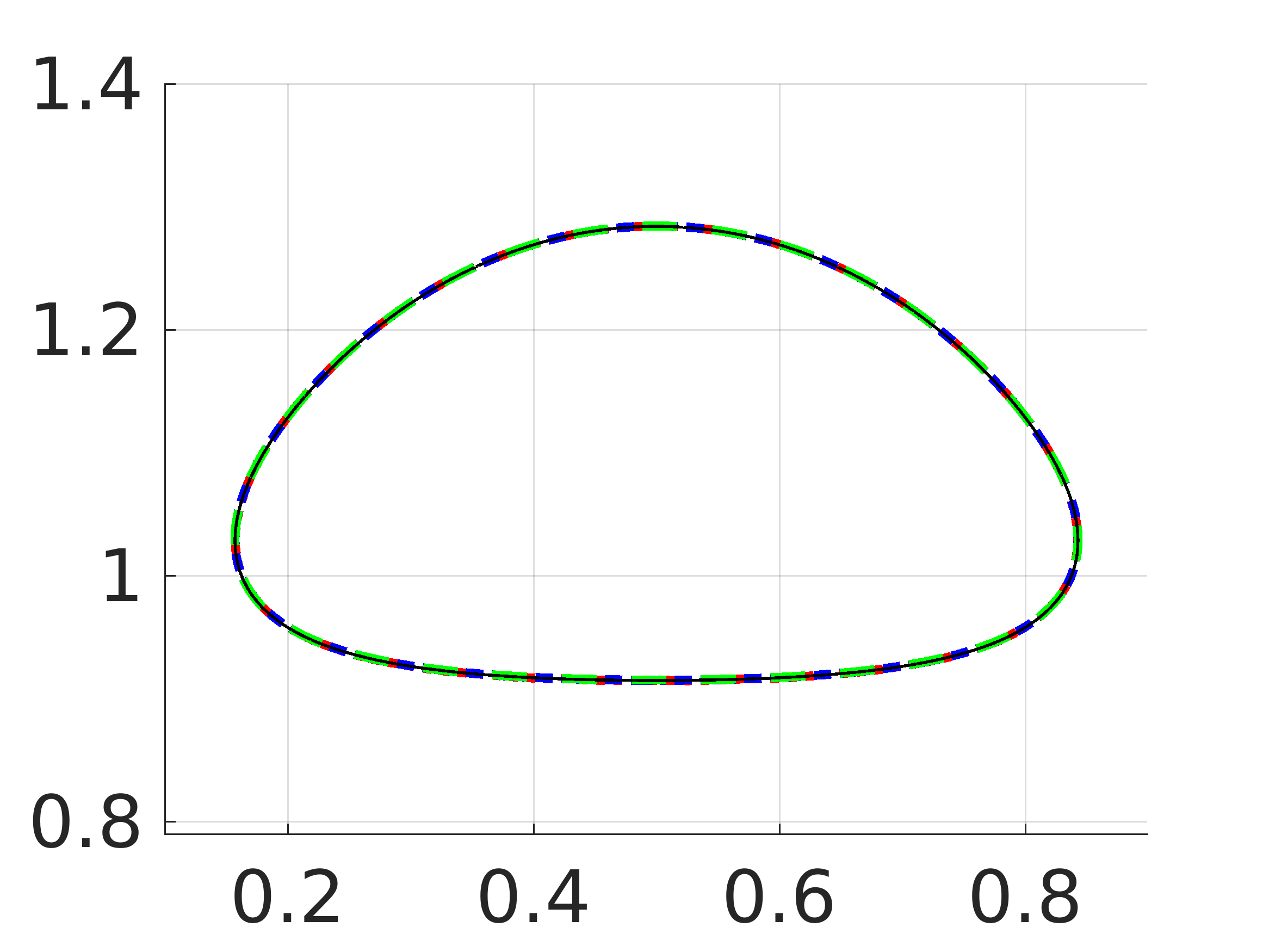}
\includegraphics[width=.4\textwidth]{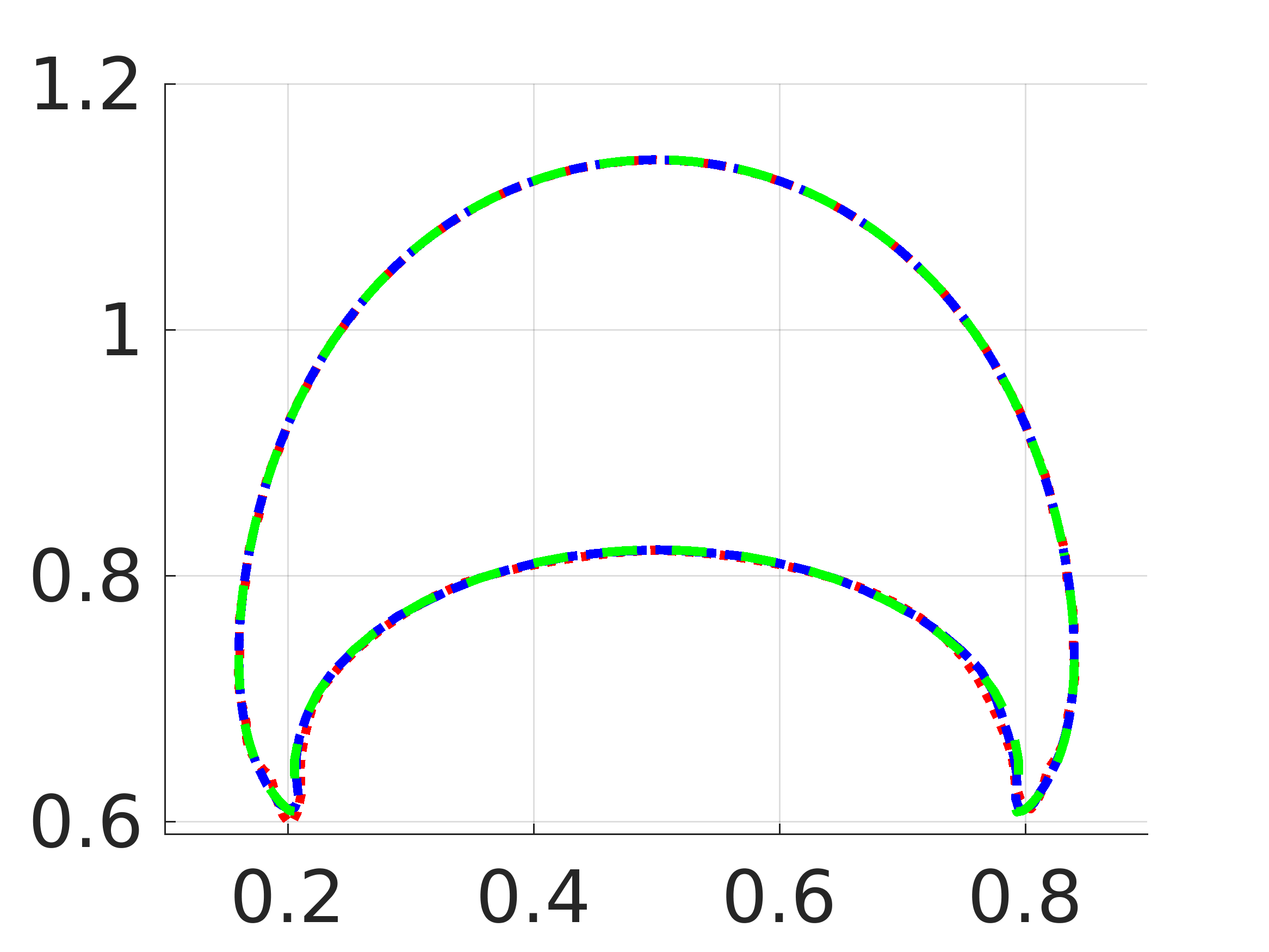}
  \caption{
    Bubble shapes at the final time. Left: Test Case 1 at time
    $t=3$. Right: Test Case 2 at time $t=2$.
    ({\sf M1}: dotted red, {\sf M2}: dashdotted blue, 
    {\sf M3}: dashed green. 
    MooNMD reference data \cite{Hysing09} (for Test Case 1): solid black.)
  }
\label{fig:bs}
\end{figure}
Zoom-in of the deformed mesh around $\Omega_2$ 
at $t=1$ and $t=2$ for Test Case 2 are shown in Fig.~\ref{fig:bd}.
We again observe no interface node clustering for all meshes.
However, due to the large deformation of the bubble shape, 
we observe self-intersection of the mesh for {\sf M1} and {\sf M2}
at time $t=2$ near the left trailing corner of the bubble.
\begin{figure}[h!]
\centering
\includegraphics[width=.28\textwidth]{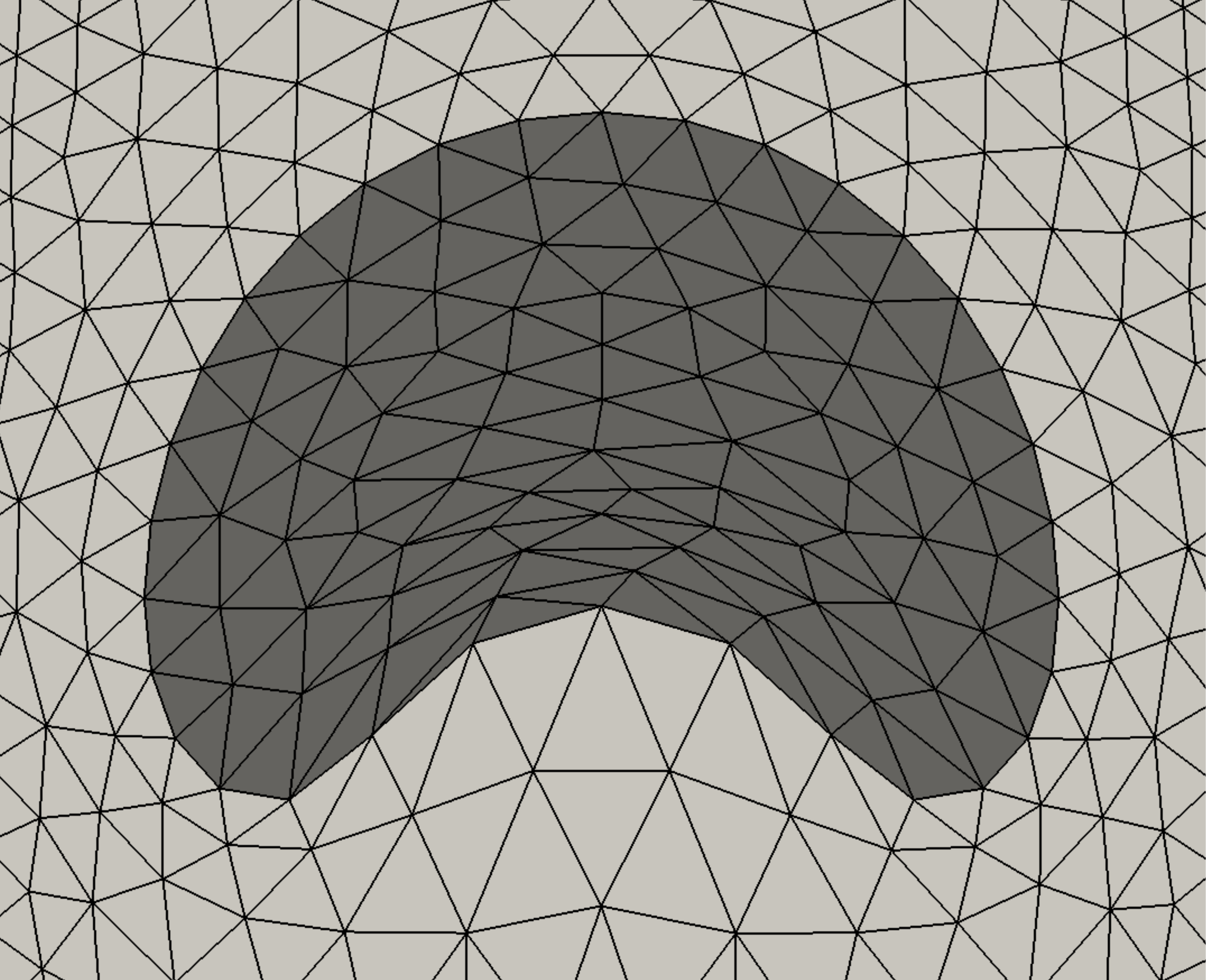}
\includegraphics[width=.28\textwidth]{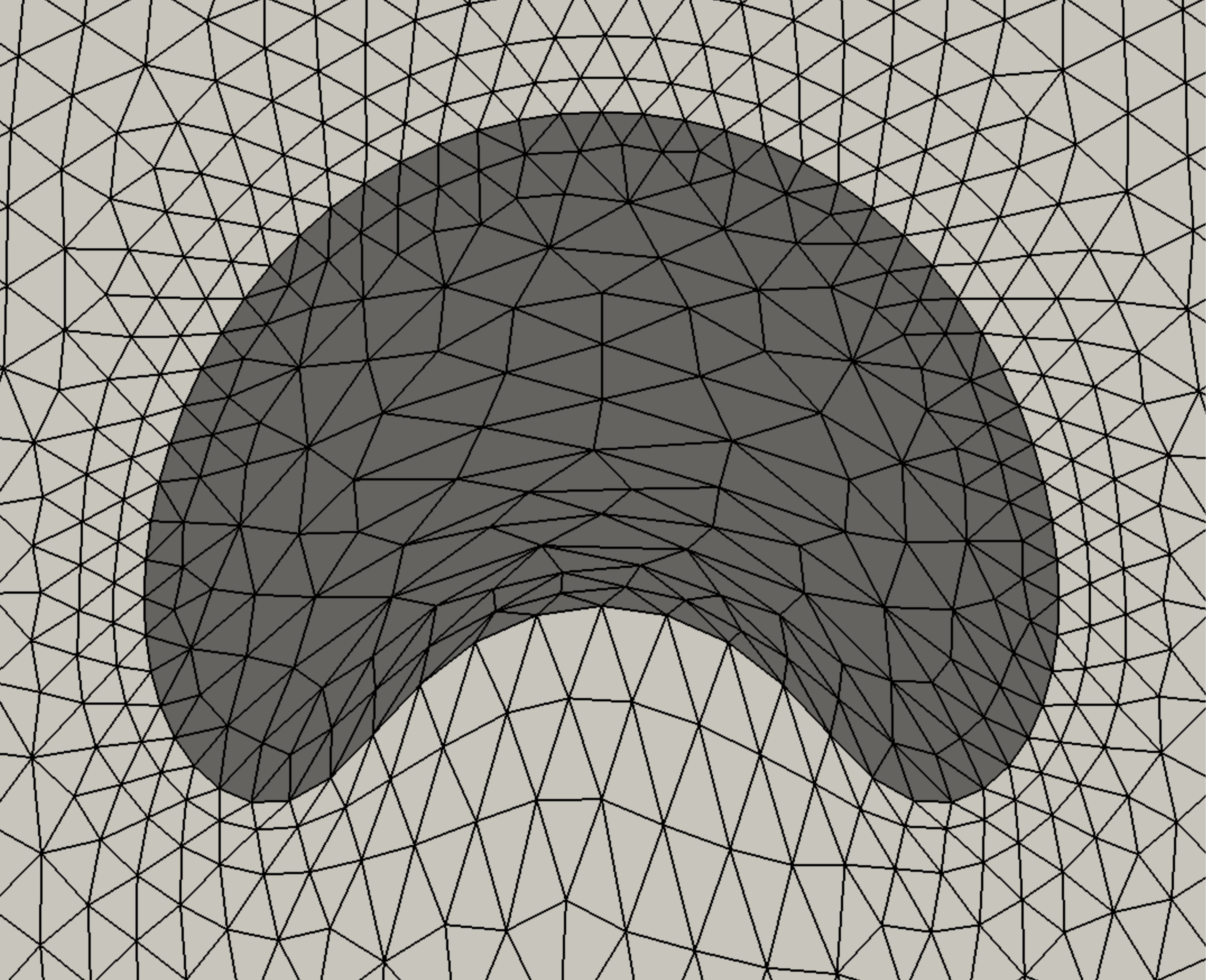}
\includegraphics[width=.28\textwidth]{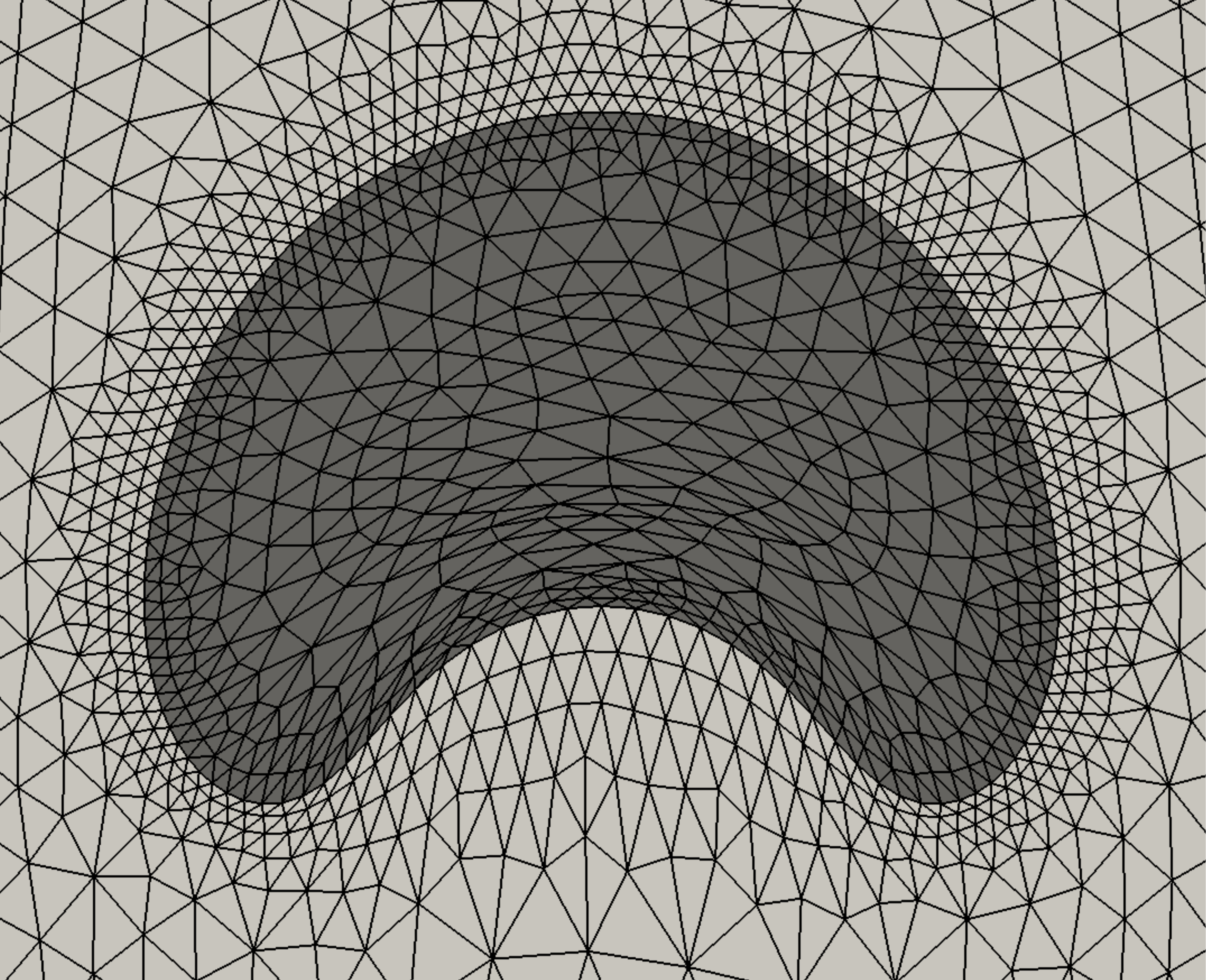}\\[.2ex]
\includegraphics[width=.28\textwidth]{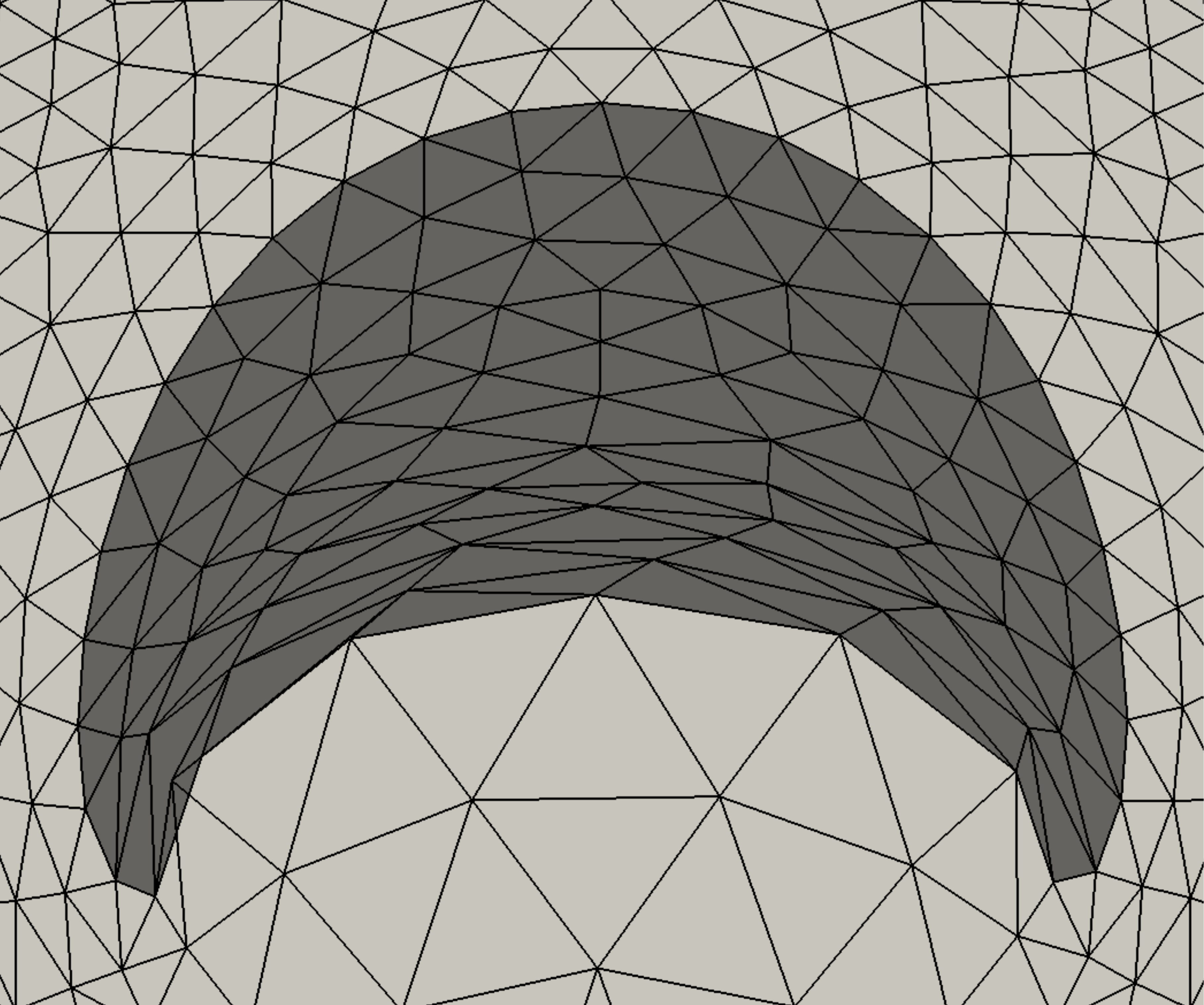}
\includegraphics[width=.28\textwidth]{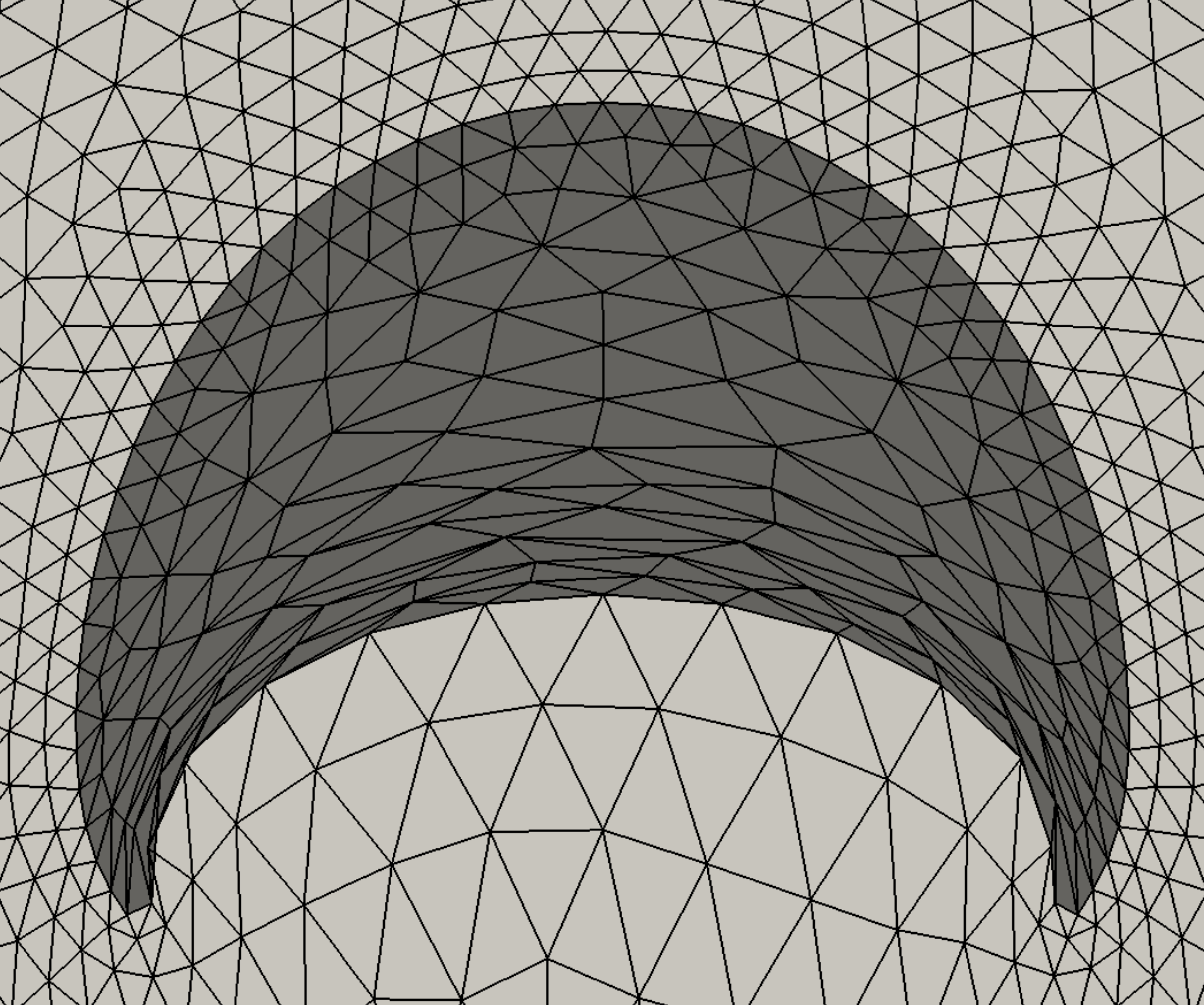}
\includegraphics[width=.28\textwidth]{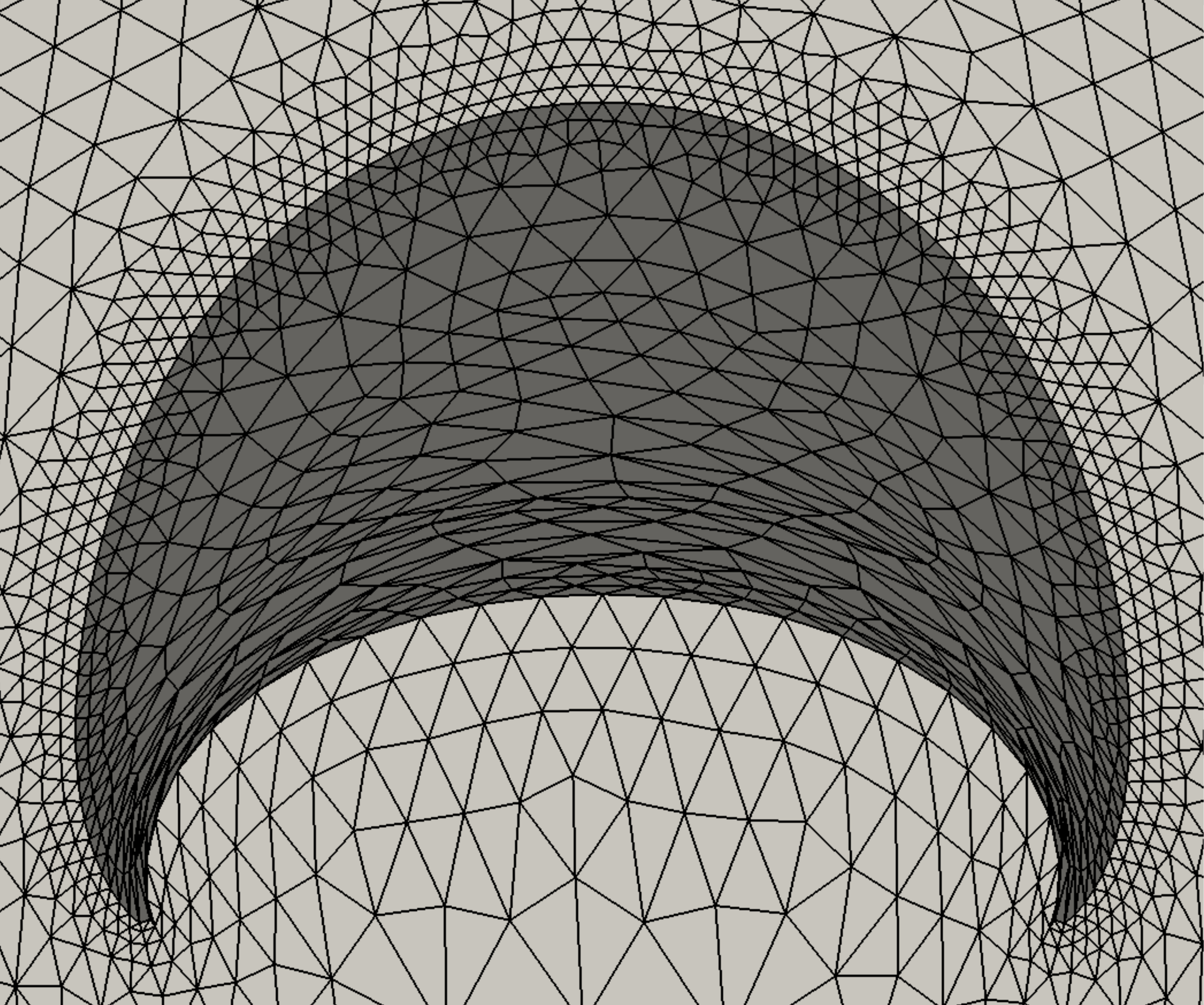}
  \caption{Zoom-in of the 
  deformed mesh around $\Omega_2$.
Left: {\sf M1}. Middle: {\sf M2}. Right: {\sf M3}.
Top: $t=1$. Bottom: $t=2$.
}
\label{fig:bd}
\end{figure}

\subsubsection{Benchmark quantities}
Following \cite{Hysing09}, we track 
the following benchmark quantities over time:
\begin{itemize}
  \item [(i)] {\it Center of mass}: the centroid of mass 
    for the bubble
    \[
      \mathbf X_c = (x_c, y_c) = 
      \frac{\int_{\Omega_2}\mathbf x\,\mathrm{dx}}{
      \int_{\Omega_2}1\,\mathrm{dx}}.
    \]
  \item [(ii)] {\it Rise velocity}: the mean bubble velocity  
    \[
      \mathbf X_c = (x_c, y_c) = 
      \frac{\int_{\Omega_2}\bld u\,\mathrm{dx}}{
      \int_{\Omega_2}1\,\mathrm{dx}}.
    \]
  \item [(iii)] {\it Circularity}: the "degree of circularity":
    \[
      \cancel{c}=\frac{P_a}{P_b}
      =\frac{\text{perimeter of area-equivalent circle}}{
      \text{perimeter of bubble}}.
    \] 
\end{itemize}

The time evolution of these benchmark quantities are shown 
in Fig.~\ref{fig:bb}. Our results are in excellent agreement with 
the MooNMD reference data \cite{Hysing09}.
\begin{figure}[h!]
\centering
\includegraphics[width=.4\textwidth]{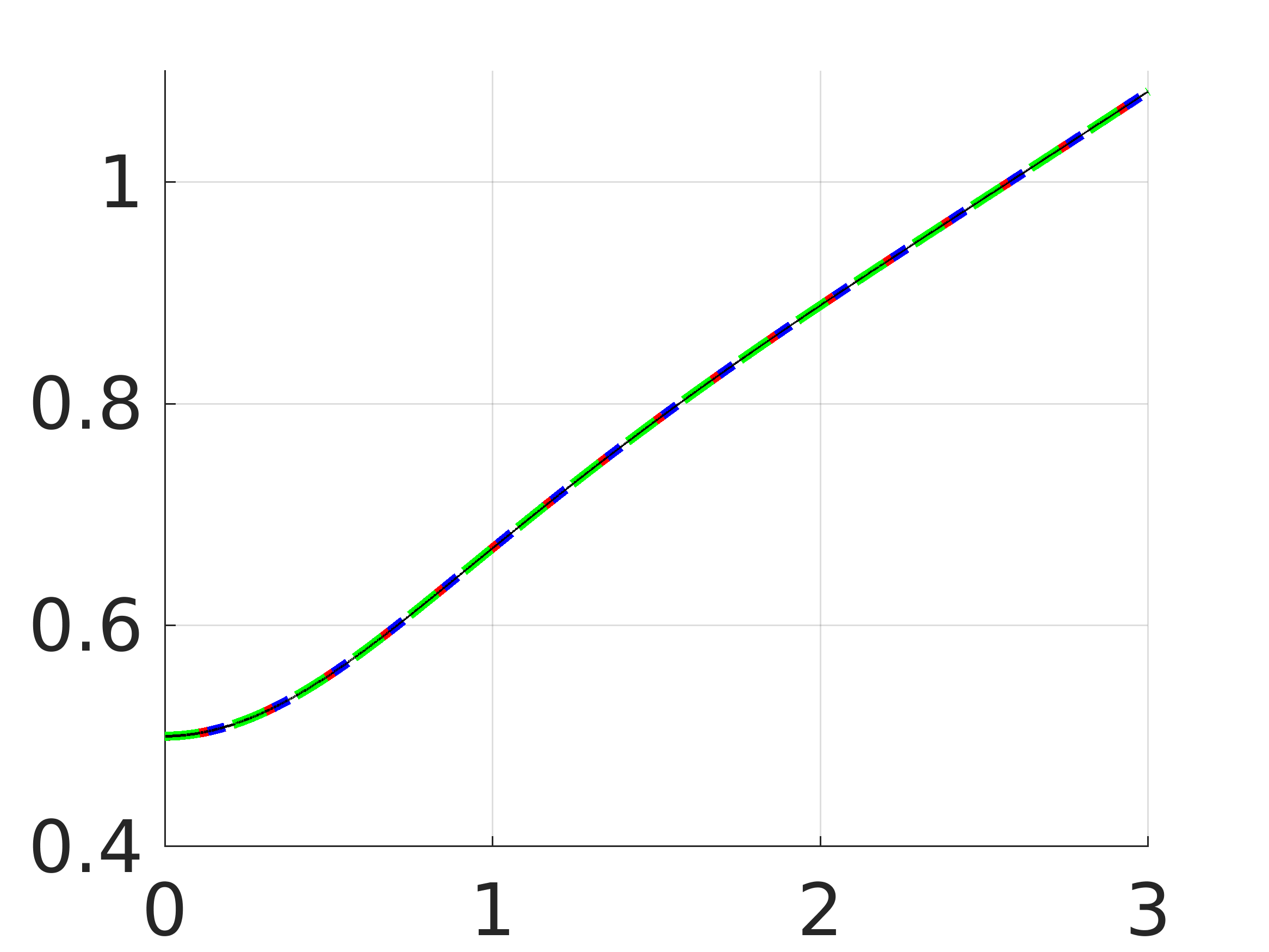}
\includegraphics[width=.4\textwidth]{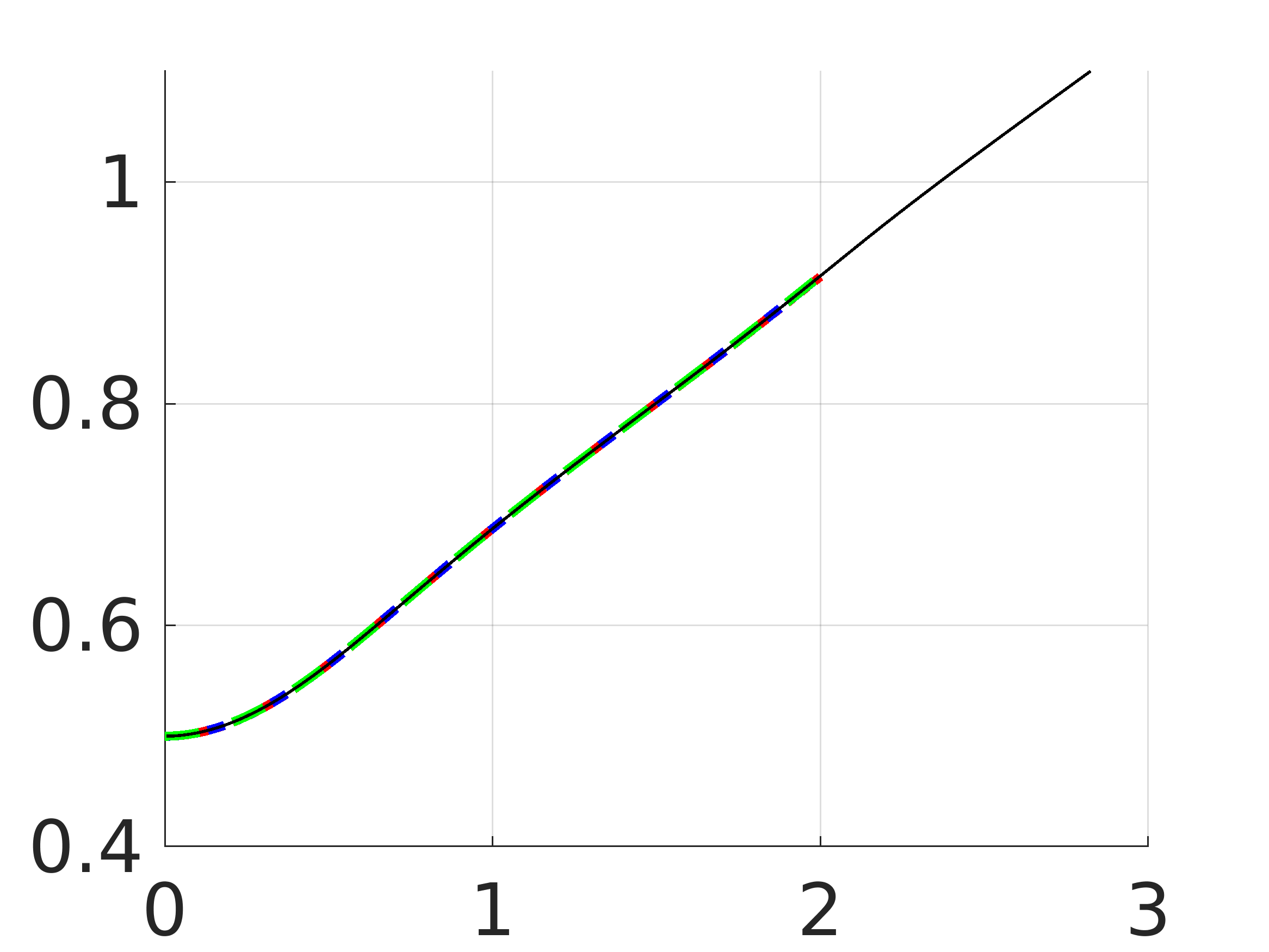}\\[.2ex]
\includegraphics[width=.4\textwidth]{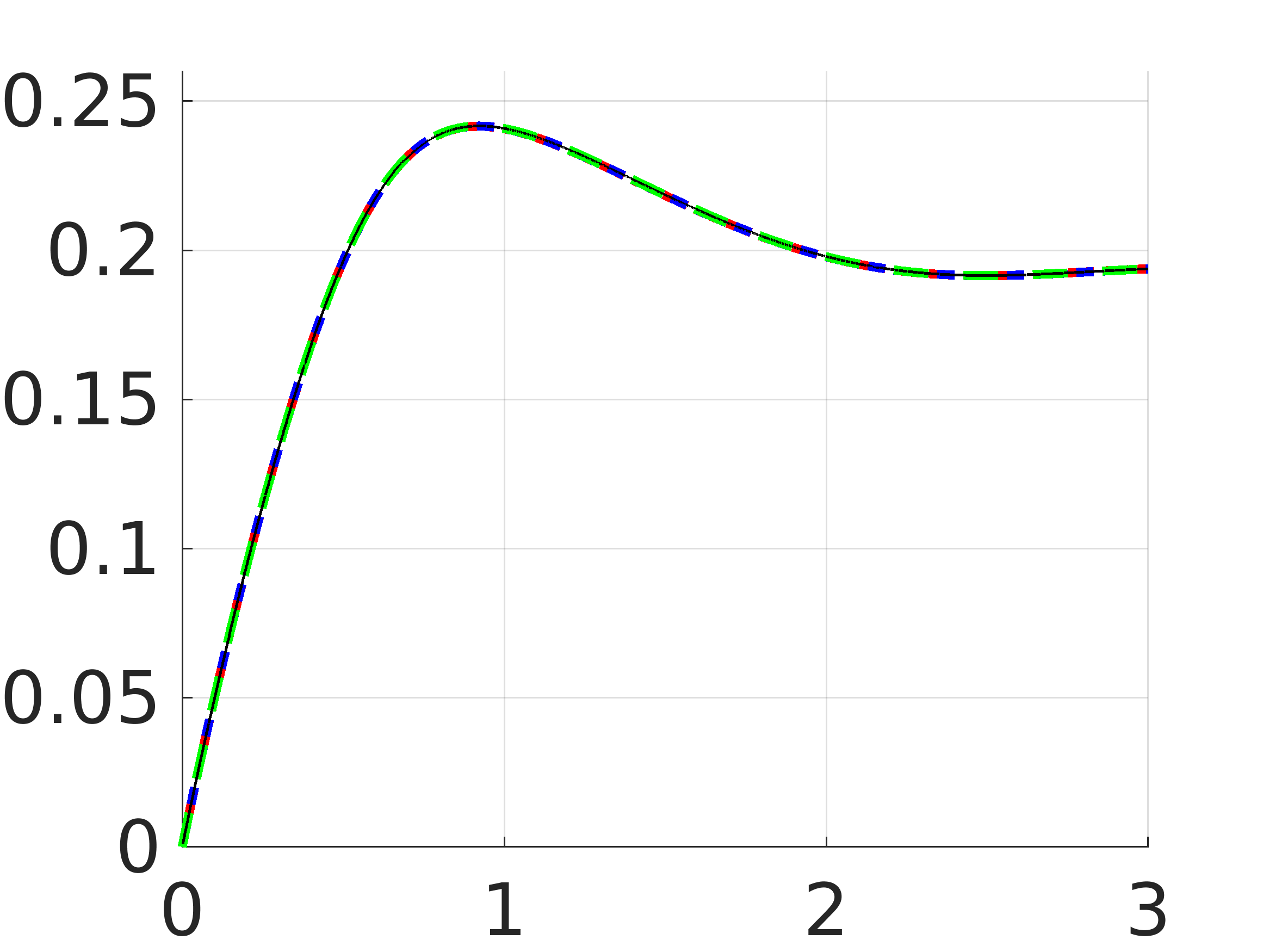}
\includegraphics[width=.4\textwidth]{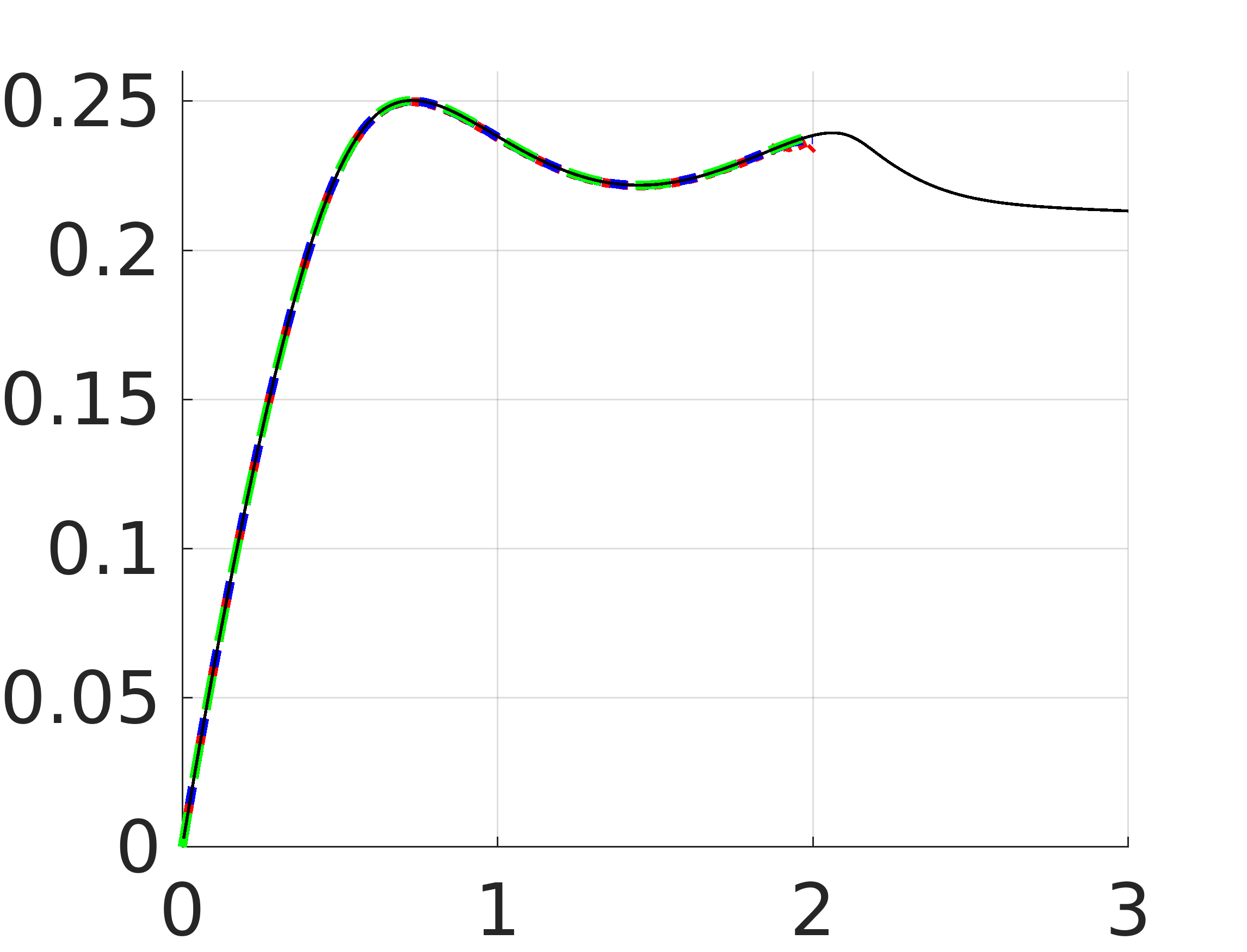}\\[.2ex]
\includegraphics[width=.4\textwidth]{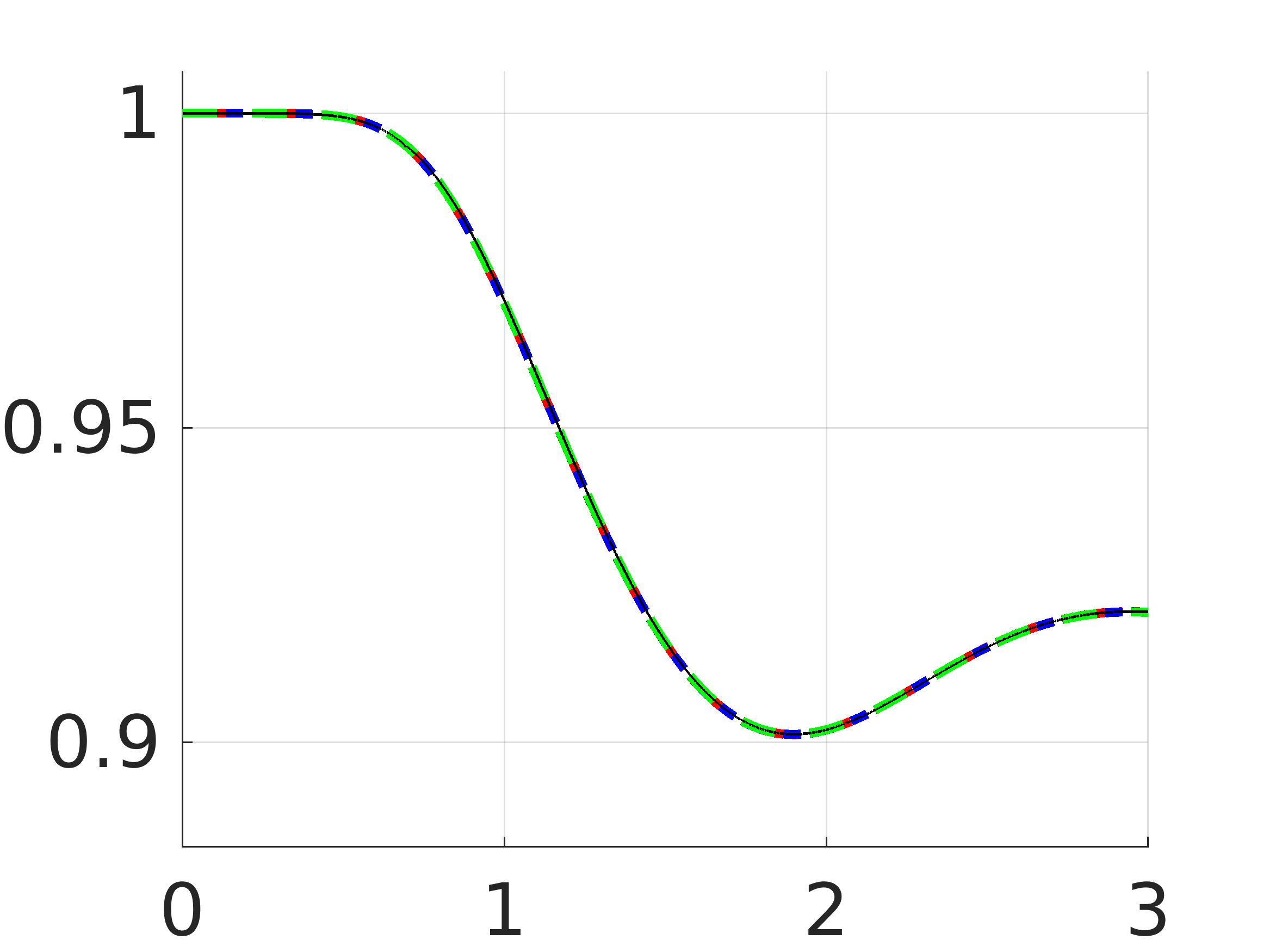}
\includegraphics[width=.4\textwidth]{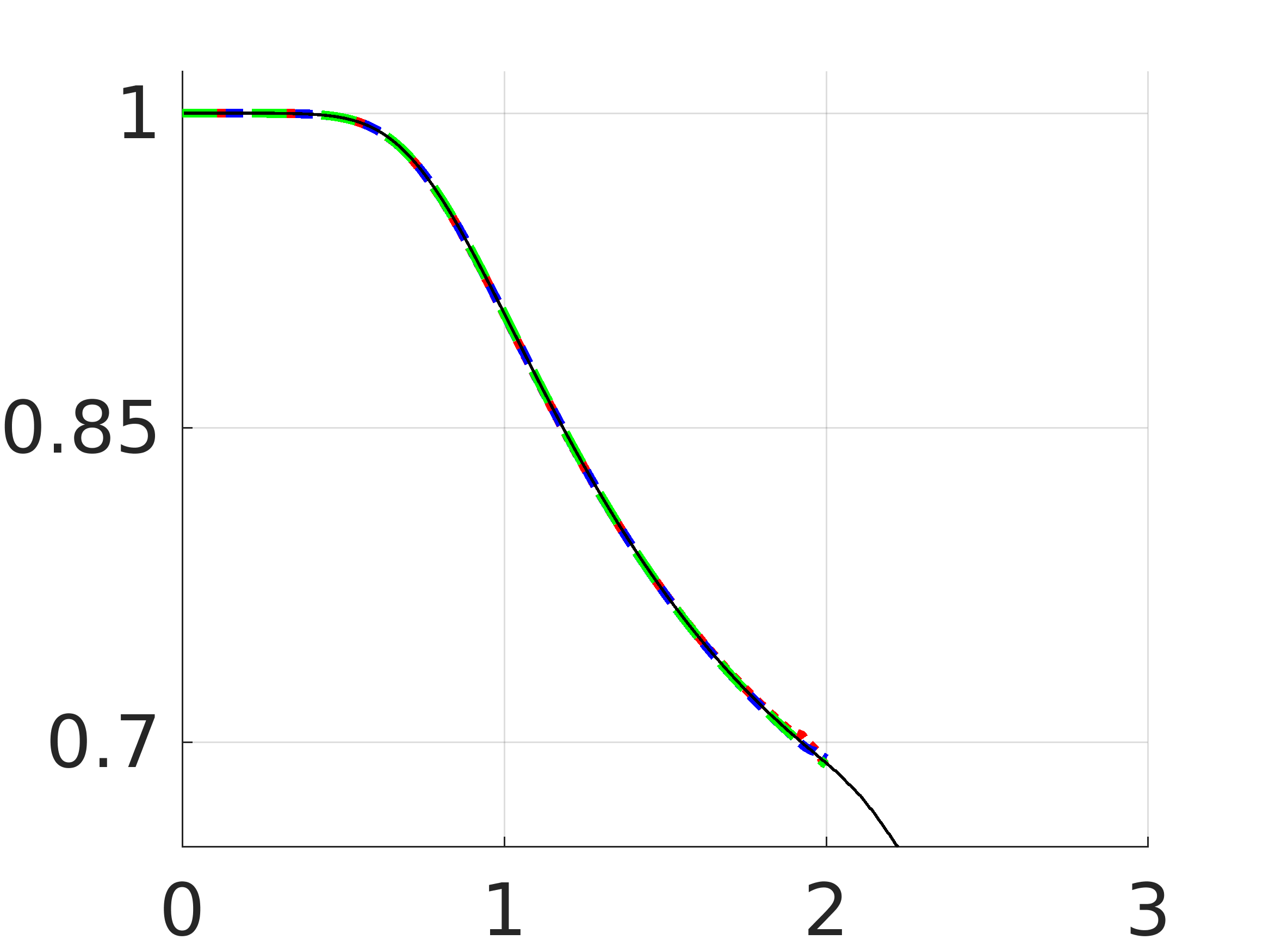}
  \caption{
  Time evolution of benchmark quantities. 
  Left: Test Case 1. Right: Test Case 2.
  Top: center of mass. Middle: rise velocity. Bottom: circularity.
  ({\sf M1}: dotted red, {\sf M2}: dashdotted blue, 
    {\sf M3}: dashed green. 
    MooNMD reference data \cite{Hysing09}: solid black.)
}
\label{fig:bb}
\end{figure}

Finally, Table~\ref{table:bb1} and Table~\ref{table:bb2} show the time and values of the minimum
circularity, maximum rise velocity, and maximum position of the center of
mass achieved during the simulations for Test Case 1 and 2, respectively.
Again, our results are in good agreement with the MooNMD reference data 
\cite{Hysing09}.
\begin{table}[ht!]
\begin{center}
 % \resizebox{\textwidth}{!} 
  {
    \begin{tabular}{lcccc} \hline
     &\hspace{1cm} {\sf M1}\hspace{1cm}  
     &\hspace{1cm} {\sf M2}\hspace{1cm}  
     &\hspace{1cm} {\sf M3}\hspace{1cm}  
     & MooNMD \cite{Hysing09}\\
     \hline
    $\cancel{c}_{\mathrm{min}}$
     & 0.90121  & 0.90123 &0.90125 &0.9013\\
     $t|_{\cancel{c}=\cancel{c}_{\mathrm{min}}}$
  & 1.9000 & 1.8969&1.8984 &1.9000\\
     ${V}_{c,\mathrm{max}}$
    & 0.24154& 0.24161 & 0.24164&0.2417\\
     $t|_{{V_c}={V}_{c,\mathrm{max}}}$
    &0.91875 & 0.92188 &0.92188&0.9239\\
     ${y_c}(t=3)$
    & 1.0817 & 1.0817 &1.0817&1.0817\\
    \hline 
\end{tabular}
}
\end{center}
\caption{
Minimum circularity and maximum rise velocity, with corresponding incidence
times, and the final position of the center of mass for Test Case 1.}
\label{table:bb1}
\end{table}

\begin{table}[ht!]
\begin{center}
 % \resizebox{\textwidth}{!} 
  {
    \begin{tabular}{lcccc} \hline
     &\hspace{1cm} {\sf M1}\hspace{1cm}  
     &\hspace{1cm} {\sf M2}\hspace{1cm}  
     &\hspace{1cm} {\sf M3}\hspace{1cm}  
     & MooNMD \cite{Hysing09}\\
     \hline
    $\cancel{c}_{\mathrm{min}}$
     & 0.68784  & 0.69210 &0.68922 &0.6901\\
     $t|_{\cancel{c}=\cancel{c}_{\mathrm{min}}}$
    & 2.0000 & 2.0000&  2.0000 &2.0000\\
     ${V}_{c,\mathrm{max}}$
    & 0.24978& 0.25001 & 0.25016&0.2502\\
     $t|_{{V_c}={V}_{c,\mathrm{max}}}$
    &0.72500 & 0.73125 &0.72969&0.7317\\
     ${y_c}(t=2)$
    & 0.91473 & 0.91524 &0.91541&0.9154\\
    \hline 
\end{tabular}
}
\end{center}
\caption{
Minimum circularity and maximum rise velocity, with corresponding incidence
times, and the final position of the center of mass for Test Case 2 (up to
time $t=2$).}
\label{table:bb2}
\end{table}

\section{Conclusion and future work}
\label{sec:conclude}
We have presented a  novel ALE-TVNNS-HDG scheme for incompressible 
flow with moving boundaries and interfaces.
Detailed theoretical analysis of the scheme consists of our ongoing work.
We are also planning to investigate 
the extension of this scheme to fluid-structure interactions.
\bibliography{fsi}
\bibliographystyle{siam}
\end{document}